\documentclass[11pt]{amsart}
\usepackage[dvipsnames]{xcolor}   		             		
\usepackage{graphicx}   	 
\usepackage{amsmath, amscd, amssymb, mathrsfs, mathabx, tikz-cd, amsthm, thmtools, stmaryrd}

\usepackage{hyperref}

\input{xy}
\xyoption{all}

\numberwithin{equation}{section}

\newcommand{\CC}{\mathbb{C}}
\newcommand{\EE}{\mathbb{E}}
\newcommand{\LL}{\mathbb{L}}

\newcommand{\PP}{\mathbb{P}}
\newcommand{\QQ}{\mathbb{Q}}
\newcommand{\RR}{\mathbb{R}}

\newcommand{\ZZ}{\mathbb{Z}}

\newcommand{\cal}{\mathcal}

\def\cC{{\cal C}}
\def\cD{{\cal D}}

\def\cM{{\cal M}}

\def\cO{{\cal O}}

\def\cU{{\cal U}}

\def\fB{\mathfrak{B}}
\def\fC{\mathfrak{C}}

\def\fM{\mathfrak{M}}

\def\fS{\mathfrak{S}}

\def\loc{\mathrm{loc}}
\def\Bl{\mathrm{Bl}}

\def\and{\quad{\rm and}\quad}
\def\lra{\longrightarrow }

\def\hra{\hookrightarrow}

\DeclareMathOperator{\image}{Im} 
\DeclareMathOperator{\id}{id}

\DeclareMathOperator{\rank}{rank}

\newtheorem{prop}{Proposition}[section]
\newtheorem{theo}[prop]{Theorem}
\newtheorem{lemm}[prop]{Lemma}
\newtheorem{coro}[prop]{Corollary}

\newtheorem{defi}[prop]{Definition}

\def\beq{\begin{equation}}
\def\eeq{\end{equation}}

\def\dual{^{\vee}}

\def\virt{^{\mathrm{vir}} }
\def\virtloc{\virt_\loc}
\def\Spec{\mathrm{Spec} }

\def\red{{\mathrm{red}}}

\def\surra{\twoheadrightarrow}

\def\hra{\hookrightarrow}

\def\bar{\overline}
\def\rou{\partial}

\def\wtil{\widetilde}
\def\what{\widehat}

\def\Pic{\mathrm{Pic}}
\def\deg{\mathrm{deg}}
\def\dim{\mathrm{dim}}

\def\Spec{\mathrm{Spec}}

\def\c>{\succ}
\def\c<{\prec}

\def\l({\left(}
\def\r){\right)}

\newcommand{\Gysin}[1]{0_{#1}^!}

\newcommand{\bdst}[1]{h^1/h^0(#1)}

\title[Algebraic reduced genus one GW invariants, Part 2]{Algebraic reduced genus one Gromov-Witten invariants for complete intersections in projective spaces, Part 2}
\author{Sanghyeon Lee}
\address{KIAS, Seoul, Korea}
\email{sanghyeon@kias.re.kr}
\author{Jeongseok Oh}
\address{KIAS, Seoul, Korea}
\email{jeongseok@kias.re.kr}

\thanks{}

\date{}

\begin{document}

\begin{abstract}
In \cite{LO18}, we provided an algebraic proof of the Zinger's comparison formula \cite{Zin09red, Zin08standard} between genus one Gromov-Witten invariants and reduced invariants when the target space is a complete intersection of dimension $2$ or $3$ in a projective space. In this paper, we extend the proof in \cite{LO18} in any dimensions and for descendant invariants.  
\end{abstract}

\maketitle
\setcounter{tocdepth}{1}
\tableofcontents

\section{Introduction}

\subsection{Main result}

Let $Q$ be a smooth projective variety over $\CC$.
For each $g,k \in \ZZ_{\geq 0}$ and $d \in H_2(Q;\ZZ)$, the moduli space of stable maps $\bar{M}_{g,k}(Q,d)$ 
carries the canonical virtual fundamental class $[\bar{M}_{g,k}(Q,d)]\virt$ of the virtual dimension $vdim:= c_1 (T_Q) \cap d  + (1-g)(\dim Q-3)+k$. 
In this article, we discuss a decomposition of $[\bar{M}_{g,k}(Q,d)]\virt$ for $g=1$ when $Q$ is embedded in $\PP^n$ as a complete intersection. 
It leads to an algebraic proof of Zinger's theorem \cite[Theorem 1A]{Zin08standard} for complete intersections in projective spaces.

Here are some notations.
\begin{itemize}
\item Let $\bar{M}^{red}_{1,k}(\PP^n,d)$ be the closure of the moduli space of stable maps with smooth domain curves $M_{1,k}(\PP^n,d) \subset \bar{M}_{1,k}(\PP^n,d)$.
\item Let $\bar{M}^{red}_{1,k}(Q,d)$ be the closed substack defined by
$$\bar{M}^{red}_{1,k}(Q,d) := \bar{M}_{1,k}(Q,d) \cap \bar{M}^{red}_{1,k}(\PP^n,d) \subset \bar{M}_{1,k}(Q,d).$$
\end{itemize}
We introduce other closed substacks in $\bar{M}_{1,k}(Q,d)$ indexed by numerical information on their rational tails.
First, we let $\fS = \fS_{k,d}$ be {\em the index set} consisting of elements
\begin{equation}\label{Sset}
\mu = ( \ (d_1(\mu),K_1(\mu)), ..., (d_{\ell(\mu)}(\mu),K_{\ell(\mu)}(\mu)) \ ),
\end{equation}
where $K_i(\mu)$ are mutually disjoint subsets of $[k]:=\{1, ..., k\}$ and $d_i(\mu)$ are positive integers with $\sum d_i(\mu) = d$.
For each $\mu \in \fS$, let $K_0(\mu)$ denote $[k] \backslash (\cup_i K_i(\mu))$.
We will abbreviate $K_0(\mu)$, $K_i(\mu)$ $d_i(\mu)$ and $\ell(\mu)$ to $K_0$, $K_i$, $d_i$ and $\ell$ when the context is clear.
For each $\mu$ of the form \eqref{Sset}, 
we assign the set $$\bar{\mu} := \{ \ (d_1(\mu),K_1(\mu)), ..., (d_{\ell(\mu)}(\mu),K_{\ell(\mu)}(\mu)) \ \},$$ 
and we define $\bar{\fS} := \{ \bar{\mu} \ | \ \mu \in \fS \}$.
\begin{itemize}
\item Let $\bar{M}_{\bar{\mu}}(\PP^n,d)$ be the closed substack of $\bar{M}_{1,k}(\PP^n,d)$ parametrizing $\bar{\mu}$-type maps; see \cite{VZ08} for the precise definition. 
\item Let $\bar{M}_{\bar{\mu}}(Q,d)$ be the closed substack defined by
$$\bar{M}_{\bar{\mu}}(Q,d) := \bar{M}_{1,k}(Q,d) \cap \bar{M}_{\bar{\mu}}(\PP^n,d).$$
\end{itemize}
We have a finite, proper node-identifying morphism \cite{VZ08}
\begin{align*} 
\iota_{\mu,Q} : \bar{\cM}_{1,K_0 \sqcup [\ell]} \times \bar{M}_{0,\bullet\sqcup K_1}(Q,d_1) \times_Q \cdots \times_Q \bar{M}_{0,\bullet\sqcup K_{\ell}}(Q,d_{\ell}) \to \bar{M}_{\bar{\mu}}(Q,d),
\end{align*}  
where $\bar{\cM}_{1,K_0\sqcup [\ell]}$ is the moduli space of genus one stable curves, and the fiber product is taken by evaluation maps of $\bullet$. Note that if $\bar{\mu}_1 = \bar{\mu}_2$, then the images of $\iota_{\mu_1,Q}$ and $\iota_{\mu_2,Q}$ are same in $\bar{M}_{1,k}(Q,d)$.
Let
$$\bar{M}_{0,\mu}(Q,d) :=  \bar{M}_{0,\bullet\sqcup K_1}(Q,d_1) \times_Q \cdots \times_Q \bar{M}_{0,\bullet\sqcup K_{\ell}}(Q,d_{\ell}) .$$ 
Note that $\bar{M}_{0,\mu}(Q,d)$ 
has the canonical virtual fundamental class with the virtual dimension $c_1(T_Q) \cap d -2\ell +\dim Q + \sum_{i=1}^{\ell} |K_i|$.

\medskip

\begin{theo} \label{main}
Suppose that $Q \subset \PP^n$ is a complete intersection in a projective space and $d \in H_2(Q;\ZZ) \to H_2(\PP^n;\ZZ) \cong \ZZ$ is a positive integer. 
Then we have a decomposition 
\begin{equation} \label{UltimateDecomp}
[\bar{M}_{1,k}(Q,d)]\virt = A^{red}_{k,d} + \sum_{\bar{\mu} \in \bar{\fS}} A^{\bar{\mu}}_{k,d}
\end{equation}
of cycles $A^{red}_{k,d}$ and $A^{\bar{\mu}}_{k,d}$ in the Chow group $A_{vdim}(\bar{M}_{1,k}(Q,d))$ 
such that
\begin{enumerate}
\item $A^{red}_{k,d}$ is supported on $\bar{M}^{red}_{1,k}(Q,d)$, and moreover, it can be written as a refined Euler class  
on the Vakil-Zinger's desingularization of $\bar{M}^{red}_{1,k}(\PP^n,d)$ \cite{VZ08}, 
\item 
$A^{\bar{\mu}}_{k,d}$ is supported on $\bar{M}_{\bar{\mu}}(Q,d)$, and moreover, its lifting along the finite morphism $\iota_{\mu,Q}$ can be expressed in terms of tautological classes, the Chern class of the tangent bundle $T_Q$ on $\bar{\cM}_{1,K_0\sqcup [\ell]} \times \bar{M}_{0,\mu}(Q,d)$. 
In particular, the integrations on $A^{\bar{\mu}}_{k,d}$ are genus zero invariants appeared in \cite[Theorem 1A]{Zin08standard}.
\end{enumerate}
\end{theo}

Throughout the paper, we consider Chow groups $A_*(-)$ with $\QQ$-coefficients.
In short, Theorem \ref{main} tells us that algebraic reduced invariants defined in \cite{CM18, CL15} coincide with reduced invariants defined in symplectic geometry \cite{Zin09red} when the target space is complete intersections in projective spaces.

\smallskip

The first property (1) of Theorem \ref{main} is known to be the {\em quantum Lefschetz property}.
The classes $A^{red}_{k,d}$ and $A^{\bar{\mu}}_{k,d}$ are already introduced with different notations in other papers \cite{CL15, CM18}.
We accordingly give the precise definitions of them in Definition \ref{Def:CycleDecomp}.
The quantum Lefschetz property for $A^{red}_{k,d}$ is already explained in authors' previous work \cite[(4.1)]{LO18}, which follows from the idea in \cite{LZ09, VZ08} (the order follows the timeline on Arxiv).
We will review it in Section \ref{Review:reducedcycle}.
The refined Euler class description of $A^{red}_{k,d}$ is equivalent to the definition in \cite{CM18}.

\smallskip

The second property (2) is precisely stated in \cite[Equation (3-29)]{Zin08standard}, which is followed by several computations in \cite[Section 3.4]{Zin08standard}.
The main effort of this article is to prove that $A^{\bar{\mu}}_{k,d}$ is written in the form of \cite[Equation (3-29)]{Zin08standard} algebraically.
On the other hand, the authors' previous work already studied certain invariants defined by integrations on $A^{\bar{\mu}}_{k,d}$ when $\dim Q = 2$ or $3$.
This investigation was easier because most integrations became zero only by dimension counts and we did not considered descendant invariants but only GW invariants.
One advantage of Theorem \ref{main} is that one can study descendant invariants. We also believe that the idea of the algebraic proof has a possibility to be applied for further works related to the quantum Lefschetz property for higher genus case \cite{LLO}.

\subsection{Desingularization and Local equations} \label{desingAndDecomp}

In this section, we review the desingularizations of moduli spaces of genus one stable maps studied by Vakil-Zinger \cite{VZ08} and Hu-Li \cite{HL10} as well as local equations of them studied by Hu-Li \cite{HL10}.

Let $\fM_{g,k}^w$ denote the moduli space of genus $g$ prestable curves with $k$-marked points, and non-negative integer weights on each component of curves.
Each object in $\fM_{g,k}^w$ is called as a weighted curve.
The moduli space $\fM_{g,k}^w$ is decomposed into $\coprod_d \fM_{g,k,d}^w$, where $d$ indicates the sum of weights on every component. 
Let $\fB_{g,k,d}$ be the stack of genus $g$ prestable curves $C$ with $k$-marked points, and line bundles $L$ of degree $d$ on $C$.
Let $\fM^{div}_{g,k,d}$ denote the stack of pairs $(C,D)$, where $C$ is a genus $g$ prestable curve with $k$-marked points, and $D$ is a degree $d$ effective divisor on $C$. 
We abbreviate subscripts on $\fM_{g,k,d}^w$ (respectively $\fB_{g,k,d}$ and $\fM^{div}_{g,k,d}$) when the context is clear.
Note that $\fM^w$, $\fB$, and $\fM^{div}$ are smooth Artin stacks.

We now assume that $g=1$.
Hu-Li constructed a (finite) sequence of blow-up $\wtil{\fM}^w \to \fM^w$ to get another smooth Artin stack; see \cite{HL10} for details.
The important property of this successive blow-up is as follows. 
Let $\wtil{\cM}:= \wtil{\fM}^w \times_{ \fM^w} \bar{M}_{1,k}(\PP^n,d)$, and $\wtil{\cM}^{\bar{\mu}} := \wtil{\fM}^{\bar{\mu}} \times_{ \fM^{\bar{\mu}}} \bar{M}_{1,k}(\PP^n,d)$ where $\fM^{\bar{\mu}}$ is a closed substack of $\fM^w$ of $\bar{\mu}$-type weighted curves and $\wtil{\fM}^{\bar{\mu}}$ is the exceptional divisor of the blow-up along its proper transform in $\wtil{\fM}^w$.
Then we have a decomposition of the space
\begin{align} \label{Decomp}
\wtil{\cM} = \wtil{\cM}^{red} \cup \bigcup_{\bar{\mu} \in \bar{\fS} }  \wtil{\cM}^{\bar{\mu}}
\end{align}
such that 
\begin{enumerate}
\item the image of $\wtil{\cM}^{red}$ under the projection morphism 
$$b: \wtil{\cM} \to  \bar{M}_{1,k}(\PP^n,d)$$ is supported on $\bar{M}^{red}_{1,k}(\PP^n,d),$
\item the image of $\wtil{\cM}^{\bar{\mu}}$ under $b$ is supported on $\bar{M}_{\bar{\mu}}(\PP^n,d)$,
\item $\wtil{\cM}^{red}$ and $\wtil{\cM}^{\bar{\mu}}$, are smooth, and they intersect transversally each other.
\end{enumerate}

Furthermore, Hu-Li described $\wtil{\cM}$ locally as a zero of a multi-valued function.
For each hyperplane $H \subset \PP^n$, we can choose an open substack $\cU_H \subset \bar{M}_{1,k}(\PP, d)$ consisting of $(C,f)$ such that $f^*(H) \subset C$ is a smooth, $d$ distinct points \cite[Section 3]{FP97}.
Note that $\cU_H$ covers $\bar{M}_{1,k}(\PP, d)$ as varying $H \subset \PP^n$.
Let $\wtil{\cU}_H := \cU_H \times_{\bar{M}_{1,k}(\PP, d)} \wtil{\cM}$.
Consider a following diagram
\begin{align} \label{Diagram:Charts}
\xymatrix{
U_H \ar@{}[rd]|{\Box} \ar[r] \ar[d] & M^{div} \ar@{}[rd]|{\Box} \ar[r] \ar[d] & B \ar@{}[rd]|{\circlearrowleft} \ar[r] \ar[d] & M^w \ar[d] \\
\wtil{\cU}_H \ar[r] & \wtil{\fM}^{div} \ar[r] & \wtil{\fB} \ar[r] & \wtil{\fM}^w,
}
\end{align}
where $\wtil{\fB} := \fB \times_{\fM^w} \wtil{\fM}^w$, and $\wtil{\fM}^{div} := \fM^{div} \times_{\fM^w} \wtil{\fM}^w$ are fiber products; $M^w \to \wtil{\fM}^w$ is a smooth affine chart; and $B = \Pic(\cC_{M^w}/M^w,d)$ is a relative Picard scheme with the weight $d$. 
Note that the fiber product of the most right-hand square of \eqref{Diagram:Charts} is a $\CC^*$-gerbe of $B$.
$M^{div}$ is a scheme since the morphism $\wtil{\fM}^{div} \to \wtil{\fB}$ is representable. 
The scheme structure on $U_H$ can be described as follows. 
Let $t = t_1 t_2 \dots t_r \in \Gamma(M^{w}, \cO_{M^{w}})$ be a product of node smoothing variables. 
Consider a multi-valued function
\begin{align*}
F : M^{div} \times \CC^n \times \CC^{dn} & \to \CC^n\\
(x,y_1,\dots,y_n,\dots) & \mapsto (t(x)y_1,\dots, t(x)y_n).
\end{align*}
There exists an open dense subset $V_H \subset M^{div} \times \CC^n \times \CC^{dn}$ such that $U_H$ is a closed substack of $V_H$ defined by $\{F|_{V_H}=0\} $. 
Locally, we have $\wtil{\cM}^{red} = \{y_1=\dots = y_n=0\}$, and $\cup_{\bar{\mu}} \wtil{\cM}^{\bar{\mu}} = \{ t=0 \} $.
For each node-smoothing parameter $t_i$, $\{t_i=0 \}$ is locally isomorphic to $\wtil{\cM}^{\bar{\mu}}$ for some $\bar{\mu} \in \bar{\fS}$.
Note that $\{t_i = 0\}$ is supported on the image of the glueing morphism
\begin{equation}\label{glueingnode}
\fM_{1,K_0 \cup [\ell]} \times \fM^w_{0,\bullet\cup K_1,d_1} \times \cdots \times \fM^w_{0,\bullet\cup K_{\ell},d_{\ell}} \to \fM^w,
\end{equation}
which is unramified and has codimension $\ell$ in $\fM^w$. Note that $\fM^{\bar{\mu}}$ is its image.

\subsection{Moduli space of stable maps to $Q$ vs Moduli space of stable maps to $\PP^n$ with fields} \label{Condition1}

Suppose that 
$$Q = \{ f_1 = \cdots = f_m = 0\} \subset \PP^n$$ 
is a complete intersection in $\PP^n$ defined by homogeneous polynomials $f_i \in H^0(\PP^n, \cO(\deg f_i))$.
Let $\pi : \cC \to \bar{M}_{1,k}(Q,d)$ be the universal curve and $ev : \cC \to Q$ be the evaluation morphism. 
The dual of the natural relative perfect obstruction theory of $\bar{M}_{1,k}(Q,d)$ over $\fM^w$ is $\RR\pi_*ev^*T_Q$. 
It gives rise to a class 
$$[\bar{M}_{1,k}(Q,d)]\virt \in A_{vdim}( \bar{M}_{1,k}(Q,d) ).$$ 
By abuse of notations, we use the same notations for universal curves and evaluation maps on any spaces when the context is clear.

Now we turn our interest to moduli space with fields.
Any morphism from a curve $C$ to $\PP^n$ can be written in terms of a line bundle $L$ on $C$ and a section $u$ in $H^0(C, L^{\oplus n+1})$.
Hence there is a forgetful morphism $\bar{M}_{1,k}(\PP^n, d) \to \fB$.
The moduli spaces of stable maps with fields $\bar{M}_{1,k}(\PP^n,d)^{p}$, constructed by Chang-Li \cite{CL11fields}, is by definition a space parametrizing $(C, L, u)  \in \bar{M}_{1,k}(\PP^n, d)$ and $p=(p_1,\dots,p_m) \in H^0(C, (\oplus_i L^{\otimes -\deg f_i}) \otimes \omega_C)$, 
where $\omega_C$ is a dualizing sheaf of $C$. 
By using the cosection-localization technique \cite{KL13cosec} studied by Kiem-Li, one can define a class $[\bar{M}_{1,k}(\PP^n,d)^{p}]\virtloc$ in $A_{vdim}(\bar{M}_{1,k}(Q,d))$.
Moreover, by \cite{CL11fields, KO18localized, CL18inv}, it is proven that
\begin{align}\label{stablevspfield}
[\bar{M}_{1,k}(\PP^n,d)^{p}]\virtloc & = (-1)^{d\sum_i \deg f_i }  [\bar{M}_{1,k}(Q,d)]\virt. 
\end{align}

Let $\wtil{\cM}^p:= \wtil{\fM}^w \times_{ \fM^w} \bar{M}_{1,k}(\PP^n,d)^p$ be the fiber product space.
Let $\wtil{\cM}_{Q}:= \wtil{\fM}^w \times_{ \fM^w} \bar{M}_{1,k}(Q,d)$ be a closed substack of $\wtil{\cM}$. 
Let $b_Q: \wtil{\cM}_{Q} \to \bar{M}_{1,k}(Q,d)$ be the projection morphism.
A pull-back of the relative perfect obstruction theory of $\bar{M}_{1,k}(\PP^n,d)^p$ over $\fB$
$$
 \left(\RR \pi_{*}ev^*\cO_{\PP^n}(1)^{\oplus (n+1)} \bigoplus \oplus_i \RR \pi_{*}(ev^*\cO_{\PP^n}(-\deg f_i) \otimes \omega_{\pi} ) \right)^{\vee}
$$
is a relative perfect obstruction theory of $\wtil{\cM}^p$ over $\wtil{\fB}$, denoted by $E_{\wtil{\cM}^p/\wtil{\fB}}$.
The localized virtual class $[\wtil{\cM}^p]\virtloc  \in A_{vdim}(\wtil{\cM}_{Q})$ defined by a pull-back of the cosection satisfies an equivalence of classes \cite[Lemma 3.2]{LO18}
\begin{align}\label{VirPushFor}
(b_Q)_*[\wtil{\cM}^p]\virtloc = [\bar{M}_{1,k}(\PP^n,d)^{p}]\virtloc. 
\end{align}
By \eqref{stablevspfield} and \eqref{VirPushFor}, we will use the left-hand side of \eqref{VirPushFor} for the proof of Theorem \ref{main}.

We can decompose $[\wtil{\cM}^p]\virtloc$ using the following decomposition of the space $\wtil{\cM}^p$. 
$\wtil{\cM}^p$ is locally defined by a zero of a multi-valued function.
More precisely, the open chart $U_H^p := U_H \times_{\wtil{\cU}_H} \wtil{\cU}_H^p$ is described by $\{F' |_{V_H^p} = 0 \}$, where $U_H$, $\wtil{\cU}_H$ are introduced in \eqref{Diagram:Charts};
$\wtil{\cU}_H^p := \wtil{\cU}_H \times_{ \wtil{\cM}}  \wtil{\cM}^p$; $V_H^p$ is a suitable open subset of $M^{div} \times \CC^n \times \CC^m \times \CC^{dn}$; and $F'$ is a multi-valued function 
\begin{align*}
 F' : M^{div} \times \CC^n \times \CC^m \times \CC^{dn} & \to \CC^n \times \CC^m \\
 (x,y_1,...,y_n,y_{n+1},...,y_{n+m},...) & \mapsto (t(x)y_1, t(x)y_2, ... , t(x)y_{n+m}).
\end{align*}

A local decomposition
$$
\{ y_1 = y_2 = \cdots =y_{n+m}= 0\} \cup \bigcup_i \{t_i=0\}
$$
gives rise to a decomposition 
\begin{align} \label{Decomp:Qp2}
\wtil{\cM}^p = \wtil{\cM}^{p, red} \cup \bigcup_{\bar{\mu} \in \bar{\fS} }  \wtil{\cM}^{p, \bar{\mu}}.
\end{align}
Note that $\wtil{\cM}^{p, red} \cong \wtil{\cM}^{red}$, which means $\wtil{\cM}^{p, red}$ is not defined by a fiber product. We will see that the virtual cycle $[\wtil{\cM}^p]$ decompose into cycles supported on components of the decomposition \eqref{Decomp:Qp2}. The localized virtual cycle $(b_Q)_*[\wtil{\cM}^p]\virtloc$ decompose into cycles supported on $\bar{M}^{red}_{1,k}(Q,d)$ and $\bar{M}_{\bar{\mu}}(Q,d)$ accordingly.

\subsection{Plan of the paper}

In Section \ref{Decomp:pfield}, we discuss decompositions of the relative intrinsic normal cones supported on the decomposition \eqref{Decomp:Qp2}.
Using one of those decompositions, we define the cycles $A^{red}_{k,d}$ and $A^{\bar{\mu}}_{k,d}$. 
In Section \ref{CoarseSpace}, we express $A^{\bar{\mu}}_{k,d}$ in terms of coarse spaces of the intrinsic normal cones in order to use local descriptions of the cones with coordinates studied in Section \ref{Decomp:Coord}.
It leads us to get a description of $A^{\bar{\mu}}_{k,d}$ in terms of Chern classes of vector bundles, which will be discussed in Section \ref{Sect:Chern}.
Finally we will prove Theorem \ref{main} in Section \ref{ChernClass} and \ref{Review:reducedcycle}.

The crucial bridge between Section \ref{Sect:virtcomp} and \ref{Sect:Pf} is Section \ref{Sect:Normal}. 
Here, we discuss how the normal bundles of node-identifying morphisms \eqref{glueingnode} are modified along the successive blow-up. 
In Section \ref{Sect:Chern} and \ref{ChernClass}, these normal bundles will be compared to the cones studied in Section \ref{Sect:virtcomp} in order to obtain a description of $A^{\bar{\mu}}_{k,d}$.

\subsection*{Acknowledgments}
The authors would like to thank Navid Nabijou for valuable comments. We also thank to the Fields Institute for wonderful working environment.

S. L. is supported by a KIAS Individual Grant MG070901 at Korea Institute for Advanced Study.
J. O. is supported by a KIAS Individual Grant MG063002 at Korea Institute for Advanced Study.

\section{Computation of normal cones}\label{Sect:virtcomp}

In this section, we study relative intrinsic normal cones of $\wtil{\cM}$ and $\wtil{\cM}^p$, and their coarse moduli spaces; see \cite{Beh09, CL15} for the definition of coarse moduli spaces of cone stacks.
More precisely, we decompose the relative intrinsic normal cone $\fC_{\wtil{\cM}/\wtil{\fM}^w}$ (resp. $\fC_{\wtil{\cM}^p/\wtil{\fM}^w}$) into irreducible components which are supported on irreducible components of $\wtil{\cM}$ (resp. $\wtil{\cM}^p$); see irreducible decompositions \eqref{Decomp} and \eqref{Decomp:Qp2}.

Using the decomposition of $\fC_{\wtil{\cM}^p/\wtil{\fM}^w}$, we will define $A^{red}_{k,d}$ and $A^{\bar{\mu}}_{k,d}$.
Also we will see that $H^1( E_{\wtil{\cM}^p/\wtil{\fM}^w }^\vee|_{\wtil{\cM}^{p, \bar{\mu}}} )$ is locally free which contains the coarse moduli space of an irreducible component of $\fC_{\wtil{\cM}^p/\wtil{\fM}^w}$ lying on $\wtil{\cM}^{p,\bar{\mu}}$ so that we can reinterpret $A^{\bar{\mu}}_{k,d}$ in terms of $H^1( E_{\wtil{\cM}^p/\wtil{\fM}^w }^\vee|_{\wtil{\cM}^{p, \bar{\mu}}} )$ and coarse moduli space of the cone.
This interpretation will be helpful for the computation of $A^{\bar{\mu}}_{k,d}$ in Section \ref{Sect:Pf}.

\subsection{Decomposition of normal cones: with local coordinates} \label{Decomp:Coord}

We compute and decompose normal cones of the moduli spaces $\wtil{\cM}$ and $\wtil{\cM}^p$ locally by using local charts and local equations referred in Section \ref{desingAndDecomp} and \ref{Condition1}.

The following is a general situation. Let $X$ be an affine scheme $X = \Spec(R)$ of a commutative $\CC$-algebra $R$ and $t=\prod\limits_i^{r}t_i \in R$ be a product element in $R$. 
Assume that both $X$ and each $\Spec (R/(t_i) )$ are irreducible. 
Consider the subscheme  
$$Y := \{ty_1 = \dots = ty_k = 0\} \subset X \times \CC^k = \Spec(R[y_1,\dots,y_k]).$$
We may consider $X = (M^{div} \times \CC^{dn}) \cap V_H$ (resp. $X = (M^{div} \times \CC^{dn}) \cap V^p_H$), $Y = U_H$ (resp. $Y = U^p_H$) and $k=n$ (resp. $k=n+m$) for a local description of $\wtil{\cM}$ (resp. $\wtil{\cM}^p$); see Section \ref{desingAndDecomp} and \ref{Condition1}.

From a direct computation, we can check that
$$
C_{Y/X\times \CC^k} \cong \Spec \left( \frac{\what{R}[x_1,\dots,x_k]}{(y_ix_j-x_jy_i)_{1\leq i<j \leq k}} \right), \  \what{R} := R[y_1,\dots,y_k]/(ty_i)_{1\leq i \leq k} \ .
$$ 
Note that $Y = \Spec(\what{R})$. Let 
$$Y^{red} := \{y_1 = \dots = y_k=0\} \text{ and } Y^i:= \{t_i = 0\} \text{ in } X \times \CC^k.$$ 
Then $X \cong Y^{red}$ and $Y = Y^{red} \cup \bigcup\limits_i \ Y^i$. 
We have 
\begin{align*}
& C_{Y/X\times \CC^k}|_{Y^{red}}  \\
& \cong \Spec\left( \left( \frac{ \what{R}[x_1,\dots,x_k] }{( y_ix_j-x_jy_i)_{1\leq i<j \leq k} } \right)
\otimes_{\what{R}} \what{R}/(y_1,\dots,y_k) \right) \\
& \cong \Spec(R[x_1,\dots,x_k]), \\
\end{align*}
and
\begin{align}\label{conelocalexpression}
& C_{Y/X\times \CC^k }|_{Y^i} \\ \nonumber
& \cong \Spec\left( \left( \frac{ \what{R}[x_1,\dots,x_k] }{( y_ix_j-x_jy_i)_{1\leq i<j \leq k} } \right)
\otimes_{ \what{R} } R[y_1,\dots,y_k]/(t_i) \right) \\ \nonumber
& \cong \Spec\left(  \frac{ R/(t_i)[y_1\dots,y_k][x_1,\dots,x_k] }{(y_ix_j-x_jy_i)_{1\leq i<j \leq k} } \right).
\end{align}
Hence $C_{Y/X\times \CC^k}|_{Y^{red}}$ is a rank $k$ vector bundle on $Y^{red} \cong X =\Spec(R)$, which is isomorphic to the normal bundle $N_{Y^{red}/X\times \CC^k}$, and $C_{Y/X\times \CC^k}|_{Y^{i}}$ is a fiber bundle over $\Spec(R/(t_i))$ whose fibers are isomorphic to the affine cone of $\Bl_{0}\CC^k$ in $\CC^k \times \CC^k$.
Since $\Spec(R)$ and $\Spec(R/(t_i))$ are irreducible, so are $C_{Y/X\times \CC^k}|_{Y^{red}}$ and $C_{Y/X\times \CC^k}|_{Y^{i}}$.

\smallskip

Now, we consider the local model of $\wtil{\cM}$.
Let 
$$\wtil{\cU}_H^{red} := \wtil{\cU}_H \times_{\wtil{\cM}} \wtil{\cM}^{red} \text{ and } \wtil{\cU}_H^{\bar{\mu}} :=  \wtil{\cU}_H \times_{\wtil{\cM}} \wtil{\cM}^{\bar{\mu}} .$$
Let $\wtil{\fM}^{rat}:= \cup_{\bar{\mu}} \wtil{\fM}^{\bar{\mu}}$.
Then $N^\vee_{\wtil{\fM}^{rat}/\wtil{\fM}^w}|_{\wtil{\cU}_H^{\bar{\mu}}}$ is locally isomorphic to the conormal bundle $\wtil{(t)/(t)^2}$ on $Y^i$, where $\bar{\mu}$-component is locally defined by $\{t_i=0\}$. 
By \cite[Proposition 3.2]{CL15}, the obstruction bundle $H^1( E_{\wtil{\cU}_H / \wtil{\fM}^{div} }\dual |_{\wtil{\cU}_H^{\bar{\mu}}} )$ is locally isomorphic to $Y^i \times \CC^n$. Since $Y$ is defined by an ideal $(ty_1,\dots,ty_n)$, we observe that the natural morphism 
\begin{align}\label{obstructionmorphism1}
N_{\wtil{\fM}^{rat}/\wtil{\fM}^w}|_{\wtil{\cU}_H^{\bar{\mu}}} \to H^1( E_{\wtil{\cU}_H / \wtil{\fM}^{div} }\dual |_{\wtil{\cU}_H^{\bar{\mu}}} )
\end{align} 
is expressed by
$$
Y^i \times \CC \to Y^i \times \CC^n, \ 1 \mapsto (y_1,\dots,y_n).
$$
From \eqref{conelocalexpression}, we also observe that 
\begin{align} \label{Remark1}
- \ \text{the cone $C_{Y/X\times \CC^k }|_{Y^i}$ is the closure of the image of the above morphism.} \ -
\end{align}

Note that $N_{\wtil{\fM}^{\bar{\mu}_i}/\wtil{\fM}^w} \dual$ is locally defined by $\wtil{ (t_i)/(t_i)^2 }$. Hence, locally we have
\begin{align*}
N_{\wtil{\fM}^{rat}/\wtil{\fM}^w}|_{\wtil{\cU}_H^{\bar{\mu}}} \stackrel{loc}{\cong} N_{\wtil{\fM}^{\bar{\mu}}/ \wtil{\fM}^w }\left( \sum\limits_{\bar{\mu}' \neq \bar{\mu}} \wtil{\fM}^{\bar{\mu}'} \cap \wtil{\fM}^{\bar{\mu}} \right) \left| _{\wtil{\cU}_H^{\bar{\mu}}} . \right.
\end{align*}
By glueing the local isomorphisms, we have a global isomorphism
\begin{align} \label{Remark2}
N_{\wtil{\fM}^{rat}/\wtil{\fM}^w}|_{\wtil{\fM}^{\bar{\mu}}} \cong N_{\wtil{\fM}^{\bar{\mu}}/ \wtil{\fM}^w }\left( \sum\limits_{  \bar{\mu}' \neq \bar{\mu} } \wtil{\fM}^{\bar{\mu}'} \cap \wtil{\fM}^{\bar{\mu}} \right).
\end{align}

\subsection{Decomposition of intrinsic normal cones of moduli spaces} \label{Decomp:pfield}

In Section \ref{Decomp:Coord}, we obtain the irreducible decomposition of the intrinsic normal cone
\begin{equation}\label{localdecomp1}
\fC_{\wtil{\cU}_H/\wtil{\fM}^{div}} = \fC_{\wtil{\cU}_H^{red}/\wtil{\fM}^{div}} \cup \left( \cup_{\bar{\mu}} \fC_{\wtil{\cU}_H /\wtil{\fM}^{div}}|_{\wtil{\cU}_H^{\bar{\mu}}} \right).
\end{equation} 
Similarly, we obtain the irreducible decomposition
\begin{equation}\label{localdecomp1:p}
\fC_{\wtil{\cU}_H^p/\wtil{\fM}^{div}} = \fC_{\wtil{\cU}_H^{p,red}/\wtil{\fM}^{div}} \cup \left( \cup_{\bar{\mu}} \fC_{\wtil{\cU}_H^p /\wtil{\fM}^{div}}|_{\wtil{\cU}_H^{p,\bar{\mu}}} \right),
\end{equation} 
where $\wtil{\cU}_H^{p, red} := \wtil{\cU}^p_H \times_{\wtil{\cM}^p} \wtil{\cM}^{p, red}$ and $\wtil{\cU}_H^{p,\bar{\mu}} :=  \wtil{\cU}^p_H \times_{\wtil{\cM}^p} \wtil{\cM}^{p, \bar{\mu}} .$
However, we cannot glue these local decompositions to get global decompositions because there is no natural morphism $\wtil{\cM} \to \wtil{\fM}^{div}$. Instead, using local decompositions \eqref{localdecomp1} and \eqref{localdecomp1:p}, we will get global decompositions of relative intrinsic normal cones $\fC_{\wtil{\cM}/\wtil{\fM}^w}$, $\fC_{\wtil{\cM}^p/\wtil{\fM}^w}$, $\fC_{\wtil{\cM}/\wtil{\fB}}$ and $\fC_{\wtil{\cM}^p/\wtil{\fB}}$.

Consider the morphism $\theta_1 : \bdst{T_{\wtil{\cU}_H/\wtil{\fM}^{div} }} \to \bdst{T_{\wtil{\cU}_H/\wtil{\fM}^w }}$ induced by an exact triangle of the tangent complexes.
By the proof of \cite[Proposition 3]{KKP}, we obtain $\theta_1^*(\fC_{\wtil{\cU}_H/\wtil{\fM}^w }) = \fC_{\wtil{\cU}_H/\wtil{\fM}^{div} }$.
Therefore the decomposition \eqref{localdecomp1} descends to the irreducible decomposition 
\begin{equation}\label{localdecomp2}
\fC_{\wtil{\cU}_H/\wtil{\fM}^w} = \fC_{\wtil{\cU}_H^{red}/\wtil{\fM}^w} \cup \left( \cup_{\bar{\mu}} \fC_{\wtil{\cU}_H /\wtil{\fM}^w}|_{\wtil{\cU}_H^{\bar{\mu}}} \right).
\end{equation} 
Similarly, \eqref{localdecomp1:p} induces the following decomposition
\begin{align}\label{localdecomp2:p}
&\fC_{\wtil{\cU}_H^p/\wtil{\fM}^{w}} = \fC_{\wtil{\cU}_H^{p,red}/\wtil{\fM}^{w}} \cup \left( \cup_{\bar{\mu}} \fC_{\wtil{\cU}_H^p /\wtil{\fM}^{w}}|_{\wtil{\cU}_H^{p,\bar{\mu}}} \right).
\end{align}
They glue to get the global decompositions of the intrinsic normal cone 
\begin{align*}
\fC_{\wtil{\cM}/\wtil{\fM}^w} & = \fC_{\wtil{\cM}^{red}/\wtil{\fM}^w} \cup \left( \cup_{\bar{\mu}} \fC_{\wtil{\cM} /\wtil{\fM}^w}|_{\wtil{\cM}^{\bar{\mu}}} \right), \\ \nonumber
\fC_{\wtil{\cM}^p/\wtil{\fM}^w} & = \fC_{\wtil{\cM}^{p,red}/\wtil{\fM}^w} \cup \left( \cup_{\bar{\mu}} \fC_{\wtil{\cM}^p /\wtil{\fM}^w}|_{\wtil{\cM}^{p,\bar{\mu}}} \right).
\end{align*}

Similarly, for an induced morphism $\theta_2 : \bdst{T_{\wtil{\cU}_H/\wtil{\fB} }} \to \bdst{T_{\wtil{\cU}_H/\wtil{\fM}^w }}$, we have $\theta_2^*(\fC_{\wtil{\cU}_H / \wtil{\fM}^w} ) = \fC_{\wtil{\cU}_H / \wtil{\fB}}$. 
Hence the local decompositions \eqref{localdecomp2} and \eqref{localdecomp2:p} pull back to the following local decompositions
\begin{align*} 
\fC_{\wtil{\cU}_H/\wtil{\fB}} & = \fC_{\wtil{\cU}_H^{red}/\wtil{\fB}} \cup \left( \cup_{\bar{\mu}} \fC_{\wtil{\cU}_H /\wtil{\fB}}|_{\wtil{\cU}_H^{\bar{\mu}}} \right), \\
\fC_{\wtil{\cU}_H^p/\wtil{\fB}} & = \fC_{\wtil{\cU}_H^{p,red}/\wtil{\fB}} \cup \left( \cup_{\bar{\mu}} \fC_{\wtil{\cU}_H^p /\wtil{\fB}}|_{\wtil{\cU}_H^{p,\bar{\mu}}} \right).
\end{align*}
They glue to the global irreducible decompositions
\begin{align*} 
\fC_{\wtil{\cM}/\wtil{\fB}} & = \fC_{\wtil{\cM}^{red}/\wtil{\fB}} \cup \left( \cup_{\bar{\mu}} \fC_{\wtil{\cM}/\wtil{\fB}} |_{\wtil{\cM}^{\bar{\mu}}}  \right), \\ \nonumber
\fC_{\wtil{\cM}^p/\wtil{\fB}} & = \fC_{\wtil{\cM}^{p,red}/\wtil{\fB}} \cup \left( \cup_{\bar{\mu}} \fC_{\wtil{\cM}^p/\wtil{\fB} } |_{\wtil{\cM}^{p,\bar{\mu}}}  \right).
\end{align*}

Now we define the classes $A^{red}_{k,d}$ and $A_{k,d}^{\bar{\mu}}$.
\begin{defi} \label{Def:CycleDecomp}
The classes $A^{red}_{k,d}, A^{\bar{\mu}}_{k,d} \in A_{vdim}(\bar{M}_{1,k}(Q,d))$ are defined by
\begin{align*}
A^{red}_{k,d} & := (-1)^{d\sum_i \deg f_i }(b_Q)_*0^!_{\bdst{E_{\wtil{\cM}^p/\wtil{\fB}} \dual}, \loc }[\fC_{\wtil{\cM}^{p,red} /\wtil{\fB}} ], \\
A^{\bar{\mu}}_{k,d} & := (-1)^{d\sum_i \deg f_i } (b_Q)_*0^!_{\bdst{E_{\wtil{\cM}^p/\wtil{\fB}} \dual}, \loc }[\fC_{\wtil{\cM}^p /\wtil{\fB}}|_{\wtil{\cM}^{p,\bar{\mu}}}].
\end{align*}
They are localized classes by a cosection; see \cite{KL13cosec} for the detail of cosection-localized class. 
The cosection will be defined in \eqref{Cosection:Def}. 
\end{defi}

By \eqref{stablevspfield}, we have
\[
[\bar{M}_{1,k}(Q,d)]\virt = (-1)^{d\sum_i \deg f_i }(b_Q)_*[\wtil{\cM}^p]\virtloc = A^{red}_{k,d} + \sum_{\bar{\mu} \in \bar{\fS}} A^{\bar{\mu}}_{k,d}.
\]
This gives the cycle decomposition \eqref{UltimateDecomp} in Theorem \ref{main}. It remains to show conditions (1) and (2) in Theorem \ref{main}. We will do this in Section \ref{Sect:Pf}.

\subsection{Coarse moduli spaces of the cone stacks} \label{CoarseSpace}
In Section \ref{Sect:Chern} and \ref{ChernClass}, we will compute $A^{\bar{\mu}}_{k,d} $ in terms of Chern classes of vector bundles.
To do so, we express $A^{\bar{\mu}}_{k,d}$ in terms of the coarse moduli spaces by using \cite[Proposition 6.3]{CL15}, namely
\begin{align}\label{coarsemoduliformula}
\Gysin{\bdst{E_{\wtil{\cM}^p/\wtil{\fB}}\dual |_{\wtil{\cM}^{p, \bar{\mu}}}},\mathrm{loc}}[\fC_{\wtil{\cM}^p/\wtil{\fB}} |_{\wtil{\cM}^{p, \bar{\mu}}}] = \Gysin{H^1(E_{\wtil{\cM}^p/\wtil{\fB}}\dual |_{\wtil{\cM}^{p, \bar{\mu}}} ),\mathrm{loc}}[C_{\wtil{\cM}^p/\wtil{\fB}} |_{\wtil{\cM}^{p, \bar{\mu}}}]
\end{align}
where $C_{\wtil{\cM}^p/\wtil{\fB}} |_{\wtil{\cM}^{p, \bar{\mu}}}$ is the coarse moduli space of the cone stack $\fC_{\wtil{\cM}^p/\wtil{\fB}} |_{\wtil{\cM}^{p, \bar{\mu}}}$.
Note that $H^1(E_{\wtil{\cM}^p/\wtil{\fB}}\dual |_{\wtil{\cM}^{p, \bar{\mu}}} )$, which is the coarse moduli space of the bundle stack $\bdst{E_{\wtil{\cM}^p/\wtil{\fB}}\dual |_{\wtil{\cM}^{p, \bar{\mu}}}}$, is a vector bundle \cite{CL15, HL10}.
In Section \ref{Sect:Chern}, we will study another formula of the right-hand side of \eqref{coarsemoduliformula} using 
$$C_{\wtil{\cM}^p/\wtil{\fM}^w}|_{\wtil{\cM}^{p, \bar{\mu}}} \subset H^1( E_{\wtil{\cM}^p/\wtil{\fM}^w }^\vee|_{\wtil{\cM}^{p, \bar{\mu}}} )$$ 
instead of 
$$C_{\wtil{\cM}^p/\wtil{\fB}}|_{\wtil{\cM}^{p, \bar{\mu}}} \subset H^1( E_{\wtil{\cM}^p/\wtil{\fB} }^\vee|_{\wtil{\cM}^{p, \bar{\mu}}} )$$ 
to have an advantage for a Chern class expression of $A^{\bar{\mu}}_{k,d}$.
Before doing this, we need to prove $H^1( E_{\wtil{\cM}^p/\wtil{\fM}^w }^\vee|_{\wtil{\cM}^{p, \bar{\mu}}} )$ is locally free.

\begin{lemm} \label{Lem:Aeq}
The coherent sheaf $H^1( E_{\wtil{\cM}^p/\wtil{\fM}^w }^\vee|_{\wtil{\cM}^{p, \bar{\mu}}} )$ is locally free of rank $n + m + \sum_i d\cdot \deg f_i$.
\end{lemm}

\begin{proof}
For a hyperplane $H \subset \PP^n$ we assign the local chart $\wtil{\cU}^p_H \subset \wtil{\cM}^p$ as in Section \ref{Condition1}. 
It is enough to show that 
$H^1( E_{\wtil{\cU}^p_H/\wtil{\fM}^w }^\vee|_{\wtil{\cU}_H^{p,\bar{\mu}}} )$
is locally free of rank $n + m + \sum_i d\cdot \deg f_i$ where $\wtil{\cU}_H^{p,\bar{\mu}}:=\{t_i = 0\} \subset \wtil{\cU}^p_H$.

Recall that the natural relative perfect obstruction theory of $\wtil{\cU}^p_H$ relative to $\wtil{\fM}^{div}$ introduced in \cite[Proposition 2.5]{CL11fields} is
$$E_{\wtil{\cU}^p_H/\wtil{\fM}^{div}}= \left(\RR\pi_*ev^* \cO_{\PP^n}(H)^{\oplus n}   \bigoplus\limits_i \RR\pi_* \left( ev^* \cO_{\PP^n}(-\deg f_i \cdot H) \otimes \omega_{\cC} \right) \right)^\vee.$$ 
Two short exact sequences
\begin{align*}
&0 \to \cO_{\PP^n} \to \cO(H)^{\oplus n+1}_{\PP^n} \to T_{\PP^n} \to 0 \\
&0 \to \cO_{\PP^n}(-H) \to \cO_{\PP^n} \to \cO_H \to 0
\end{align*}  
give rise to an exact triangle
\begin{align} \label{ExactTri}
\RR\pi_* ev^* \cO_{\PP^n}(H)^{\oplus n} \to \RR\pi_* ev^* T_{\PP^n} \to \RR\pi_* ev^*\cO_H (H) \stackrel{+1}{\lra} .
\end{align}
Since the effective divisor $ev^*H \subset C$ is the sum of distinct, smooth points, $\RR\pi_* ev^*\cO_{H}(H) \cong \pi_H^*T_{\wtil{\fM}^{div}/\wtil{\fM}^w }$, where $\pi_H: \wtil{\cU}^p_H \to \wtil{\fM}^{div}$ is the projection morphism. 
Thus, we have the following diagram of triangles
\begin{align} \label{POTdiagram1}
\xymatrix{
\RR\pi_* ev^*\cO_{H}(H)[-1] \ar[r] & E_{\wtil{\cU}^p_H/\wtil{\fM}^{div} }\dual \ar[r] & E_{\wtil{\cU}^p_H/\wtil{\fM}^w }\dual \ar[r]^(0.6){+1} &  \\
\pi_H^* T_{\wtil{\fM}^{div}/\wtil{\fM}^w}[-1] \ar[r] \ar[u]^{\cong} & T_{\wtil{\cU}^p_H/\wtil{\fM}^{div} } \ar[r] \ar[u] & T_{\wtil{\cU}^p_H/\wtil{\fM}^w } \ar[u] \ar[r]^(0.6){+1} &,  
}
\end{align}
where two right-hand vertical morphisms are given by dual of relative perfect obstruction theories. The first row of the diagram \eqref{POTdiagram1} comes from \eqref{ExactTri}. 
If we restrict the above sequence in \eqref{POTdiagram1} to $\wtil{\cU}_H^{p,\bar{\mu}}$ and take the first cohomology, then we have the following exact sequence
\begin{align}\label{obstructioncompare1}
\to \pi_*ev^*\cO_{H}(H) \to H^1( E_{\wtil{\cU}^p_H/\wtil{\fM}^{div} }^\vee|_{\wtil{\cU}^{p,\bar{\mu}}_{H}} ) \to H^1( E_{\wtil{\cU}^p_H/\wtil{\fM}^w }^\vee|_{\wtil{\cU}_H^{p,\bar{\mu}}} ) \to 0.
\end{align}
Since we have
\begin{align*}
& E_{\wtil{\cU}^p_H/\wtil{\fM}^{div} }^\vee|_{\wtil{\cU}_H^{p,\bar{\mu}}} \\ 
&  \cong \left( \RR \pi_{*}ev^*\cO_{\PP^n}(H)^{\oplus n} \oplus \bigoplus\limits_i \RR\pi_* \left( ev^* \cO_{\PP^n}(-\deg f_i \cdot H) \otimes \omega_{\cC} \right) \right) \left|_{ \; \wtil{\cU}_H^{p,\bar{\mu}}} \right. \\
& \cong [\cO_{\wtil{\cU}_H^{p,\bar{\mu}}}^{\oplus d+1} \stackrel{0}{\lra} \cO_{\wtil{\cU}_H^{p,\bar{\mu}}}  ]^{\oplus n} \oplus \bigoplus_i [\cO_{\wtil{\cU}_H^{p,\bar{\mu}}} \stackrel{0}{\lra} \cO_{\wtil{\cU}_H^{p,\bar{\mu}}}^{\oplus (d\cdot \deg f_i + 1)}  ]
\end{align*}
from \cite{CL15,HL10}, we have $H^1( E_{\wtil{\cU}^p_H/\wtil{\fM}^{div} }^\vee|_{\wtil{\cU}_H^{p,\bar{\mu}}} ) \cong \cO_{\wtil{\cU}_H^{p,\bar{\mu}}}^{\oplus \left( n + m + \sum_i d\cdot \deg f_i \right)}$. 
Thus we have $\rank(H^1( E_{\wtil{\cU}^p_H/\wtil{\fM}^w }^\vee|_{\wtil{\cU}_H^{p,\bar{\mu}}} )) \leq n + m + \sum_i d\cdot \deg f_i$.

Consider the following diagram in the proof of \cite[Lemma 2.8]{CL11fields}
\begin{align} \label{POTdiagram2}
\xymatrix{
R\pi_* \cO_{\cC} \ar[r] & E_{\wtil{\cU}^p_H/\wtil{\fB} }\dual \ar[r] & E_{\wtil{\cU}^p_H/\wtil{\fM}^w }\dual \ar[r]^(0.6){+1} &  \\
\pi_{\wtil{\fB}}^* T_{\wtil{\fB}/\wtil{\fM}^w}[-1] \ar[r] \ar[u]^{\cong} & T_{\wtil{\cU}^p_H/\wtil{\fB} } \ar[r] \ar[u] & T_{\wtil{\cU}^p_H/\wtil{\fM}^w } \ar[u] \ar[r]^(0.6){+1} &,  
}
\end{align}
where $\pi_{\wtil{\fB}} : \wtil{\cU}^p_H \to \wtil{\fB}$ is the natural projection morphism.
If we restrict the above sequence in \eqref{POTdiagram2} to $\wtil{\cU}_H^{p, \bar{\mu}}$ and take the first cohomology, then we have the following exact sequence:
\begin{align*} 
 \to R^1\pi_*\cO_{\cC} \to H^1( E_{\wtil{\cU}^p_H/\wtil{\fB} }^\vee|_{\wtil{\cU}_H^{p, \bar{\mu}}} ) \to H^1( E_{\wtil{\cU}^p_H/\wtil{\fM}^w }^\vee|_{\wtil{\cU}_H^{p, \bar{\mu}}} ) \to 0.
\end{align*}
Since we have 
\begin{align*}
& E_{\wtil{\cU}^p_H/\wtil{\fB} }^\vee|_{\wtil{\cU}_H^{p,\bar{\mu}}} \\
&  \cong \left( \RR \pi_{*}ev^*\cO_{\PP^n}(H)^{\oplus (n+1)} \oplus \bigoplus\limits_i \RR\pi_* \left( ev^* \cO_{\PP^n}(-\deg f_i \cdot H) \otimes \omega_{\cC} \right) \right) \left|_{ \; \wtil{\cU}_H^{p,\bar{\mu}}} \right. \\
& \cong [\cO_{\wtil{\cU}_H^{p,\bar{\mu}}}^{\oplus d+1} \stackrel{0}{\lra} \cO_{\wtil{\cU}_H^{p,\bar{\mu}}}  ]^{\oplus (n+1)} \oplus \bigoplus_i [\cO_{\wtil{\cU}_H^{p,\bar{\mu}}} \stackrel{0}{\lra} \cO_{\wtil{\cU}_H^{p,\bar{\mu}}}^{\oplus (d\cdot \deg f_i + 1)}  ],
\end{align*} 
we conclude that $\rank(H^1( E_{\wtil{\cU}^p_H/\wtil{\fM}^w }^\vee|_{\wtil{\cU}_H^{p, \bar{\mu}}} )) \geq n + m + \sum_i d\cdot \deg f_i$. 
It implies that $H^1( E_{\wtil{\cU}^p_H/\wtil{\fM}^w }^\vee|_{\wtil{\cU}_H^{p, \bar{\mu}}} )$ is a vector bundle of rank $n + m + \sum_i d\cdot \deg f_i$. 
\end{proof}

Similarly, the coherent sheaf $H^1( E_{\wtil{\cM}/\wtil{\fM}^w }^\vee|_{\wtil{\cM}^{ \bar{\mu}}} )$ is locally free of rank $n $.
The obstruction bundles $H^1( E_{\wtil{\cU}^p_H/\wtil{\fM}^{div} }^\vee|_{\wtil{\cU}^{p,\bar{\mu}}_{H}} )$ and $H^1( E_{\wtil{\cU}^p_H/\wtil{\fM}^w }^\vee|_{\wtil{\cU}_H^{p,\bar{\mu}}} )$ are isomorphic since they have the same rank in the exact sequence \eqref{obstructioncompare1}.
The bundles $H^1( E_{\wtil{\cU}_H/\wtil{\fM}^{div} }^\vee|_{\wtil{\cU}^{\bar{\mu}}_{H}} )$ and $H^1( E_{\wtil{\cU}_H/\wtil{\fM}^w }^\vee|_{\wtil{\cU}_H^{\bar{\mu}}} )$ are also isomorphic.

\begin{coro}\label{obstructionmorphism2}
We can glue the morphism \eqref{obstructionmorphism1}  to obtain
\[
N_{\wtil{\fM}^{rat}/\wtil{\fM}^w}|_{\wtil{\cM}^{\bar{\mu}}} \to H^1(E_{\wtil{\cM}/ \wtil{\fM}^w}\dual|_{\wtil{\cM}^{\bar{\mu}}}) 
\]
whose zero locus is $\wtil{\cM}^{\bar{\mu}} \cap \wtil{\cM}^{red}$.
\end{coro}

By \eqref{POTdiagram2}, we obtain a diagram of short exact sequences of abelian cones
\begin{align}\label{SESofH1:2}
\xymatrix{
 0 \ar[r] & R^1\pi_*\cO_{\cC_{\wtil{\cU}_H^{p, \bar{\mu}}} } \ar[r] & H^1( E_{\wtil{\cU}^p_H/\wtil{\fB} }^\vee|_{\wtil{\cU}_H^{p, \bar{\mu}}} ) \ar[r] &  H^1( E_{\wtil{\cU}_H^p/\wtil{\fM}^w }^\vee|_{\wtil{\cU}_H^{p, \bar{\mu}}} ) \ar[r] & 0 \\
 0 \ar[r]  & R^1\pi_*\cO_{\cC_{\wtil{\cU}_H^{p, \bar{\mu}}} } \ar[r] \ar@{=}[u]  & C_{\wtil{\cU}^p_H/\wtil{\fB}} |_{\wtil{\cU}_H^{p, \bar{\mu}}}
 \ar[r] \ar@{^(->}[u] & C_{\wtil{\cU}^p_H/\wtil{\fM}^w} |_{\wtil{\cU}_H^{p, \bar{\mu}}}  \ar[r] \ar@{^(->}[u] & 0 ,
 }
\end{align}
where $C_{\wtil{\cU}^p_H/\wtil{\fB}} |_{\wtil{\cU}_H^{p, \bar{\mu}}}$ (resp. $C_{\wtil{\cU}^p_H/\wtil{\fM}^w} |_{\wtil{\cU}_H^{p, \bar{\mu}}}$) is the coarse moduli space of the cone stack $\fC_{\wtil{\cU}^p_H/\wtil{\fB}} |_{\wtil{\cU}_H^{p, \bar{\mu}}}$ (resp. $\fC_{\wtil{\cU}^p_H/\wtil{\fM}^w} |_{\wtil{\cU}_H^{p, \bar{\mu}}}$).
We glue these local exact sequences to obtain a global fiber diagram 
\[
\xymatrix{
C_{\wtil{\cM}^p/\wtil{\fB}} |_{\wtil{\cM}^{p, \bar{\mu}}} \ar@{^(->}[r] \ar[d] & H^1( E_{\wtil{\cM}^p/\wtil{\fB} }^\vee|_{\wtil{\cM}^{p, \bar{\mu}}} ) \ar@{->>}[d]^{\theta} \\
C_{\wtil{\cM}^p_H/\wtil{\fM}^w} |_{\wtil{\cM}_H^{p, \bar{\mu}}} \ar@{^(->}[r] & H^1( E_{\wtil{\cM}^p/\wtil{\fM}^w }^\vee|_{\wtil{\cM}^{p, \bar{\mu}}} )
}
\]
where $\theta$ is the morphism of vector bundles induced by $\theta_2$.
In other words, we have $\theta^*(C_{\wtil{\cM}^p_H/\wtil{\fM}^w} |_{\wtil{\cM}_H^{p, \bar{\mu}}}) = C_{\wtil{\cM}^p/\wtil{\fB}} |_{\wtil{\cM}^{p, \bar{\mu}}}$.

\section{Normal bundles of node-identifying morphisms} \label{Sect:Normal}

In this section, we give some analogues of the results in \cite[Section 2,3 and 4]{VZ08} and \cite[Section 3.4]{Zin08standard}. 
Here is the summary of what we will do in this section.
For a fixed $\mu \in \fS$, we consider the following fiber diagram of the node-identifying morphism $\iota_{\mu}$ and the blow-up morphism $\wtil{\fM}^w \to \fM^w$ 
\[
\xymatrix{ 
\fM_{\mathrm{fib}} \ar[r] \ar@{}[rd]|{\Box} \ar[r]^-{\wtil{\iota}_{\mu}} \ar[d]  & \wtil{\fM}^w \ar[d] \\
\bar{\cM}_{1,K_0 \cup [\ell]} \times \fM^w_{0,\bullet\cup K_1,d_1} \times \cdots \times \fM^w_{0,\bullet\cup K_{\ell},d_{\ell}} \ar[r]^-{\iota_{\mu}} & \fM^w
.}
\]
We will express the space $\fM_{\mathrm{fib}}$ precisely later.
Then  
\begin{align} \label{Comp:NormalB}
-\  \text{the normal bundle $N_{\wtil{\iota}_{\mu}}$ is obtained by some modifications of $N_{\iota_{\mu}}$}. \ -
\end{align}
Note that $N_{\wtil{\iota}_{\mu}} \cong \wtil{\iota}^*_\mu N_{\wtil{\fM}^{\bar{\mu}}/\wtil{\fM}^w}$. 
It will be combined with \eqref{Remark1}, \eqref{Remark2}, Corollary \ref{obstructionmorphism2}, and the explicit description of $N_{\iota_{\mu}}$ \eqref{Normal:First} below through \eqref{Comp:NormalB} for the computation of $A^{\bar{\mu}}_{k,d}$.
We fix an element 
$$\bar{\mu} = \{ (d_1,K_1),\dots, (d_{\ell}, K_{\ell}) \} \in \bar{\fS}.$$
Note that $\iota_{\mu}$ is unramified finite morphism, but not an embedding. Moreover it factors through the morphism \eqref{glueingnode}.
The image of $\iota_{\mu}$ is an open substack of $\fM^{\bar{\mu}}$ which is the image of \eqref{glueingnode}.
So the relative normal bundle of $\iota_{\mu}$ is 
\begin{align} \label{Normal:First}
N_{\iota_{\mu}} = \oplus_{i\in [\ell]} L_i \boxtimes L_{\bullet,i}
\end{align}
where $L_i$ is the $i$-th tautological bundle on $\bar{\cM}_{1,K_0\sqcup [\ell] }$ and $L_{\bullet,i}$ is the tautological bundle on $\fM^w_{0,\bullet\sqcup K_i, d_i}$ corresponding to the marked point $\bullet$ of $\fM^w_{\bullet\sqcup K_i,d_i}$.

\subsection{Successive blow-up and normal bundles}\label{Succblowup}
Now, we discuss the detail of \eqref{Comp:NormalB}.
First we want to describe a successive blow-up $\wtil{\fM}^w \to \fM^w$, and explain how the relative normal bundle $N_{\iota_{\mu}}$ is modified along the blow-up later.
We introduce a partial order on $\bar{\fS}$ following \cite[(2-2)]{VZ08}. 
For $\bar{\mu}_a, \bar{\mu}_b \in \bar{\fS}$, 
$$\bar{\mu}_a \c< \bar{\mu}_b \iff
\fM^{\bar{\mu}_a} \cap \fM^{\bar{\mu}_b} \neq \varnothing, \ l(\mu_a) + |K_0(\mu_a)| < l(\mu_b) + |K_0(\mu_b)|.$$ 
We then fix any complete order $<$ on $\bar{\fS}$ an extension of this partial order. Namely, 
$$\bar{\fS} = \{ \bar{\mu}_1 < \bar{\mu}_2 < \dots < \bar{\mu}=\bar{\mu}_{N(\bar{\mu})} < \bar{\mu}_{N(\bar{\mu})+1} < \dots < \bar{\mu}_N  \}.$$ 
We denote $\fM^{|0} := \fM^w$, which is an initial space of the sequence of blow-ups.
Let $\pi_1 : \fM^{|1} := \Bl_{\fM^{\bar{\mu}_1}} \fM^{|0} \to \fM^{|0}$ be the blow-up morphism, and
$\fM^{\bar{\mu}_j|1}$ be the proper transforms of $\fM^{\bar{\mu}_j}$ via $\pi_1$. 
Inductively we define the blow-up morphism $\pi_i : \fM^{|i}:=\Bl_{\fM^{\bar{\mu}_i|i-1}} \fM^{|i-1} \to \fM^{|i-1}$ and the proper transforms $\fM^{\bar{\mu}_j|i}$ of $\fM^{\bar{\mu}_j|i-1}$ via $\pi_i$. 
In the final step, we obtain $\wtil{\fM}^w : = \fM^{|N}$ and $\wtil{\fM}^{\bar{\mu}_j} := \fM^{\bar{\mu}_j|N}$. Let $\wtil{\pi} : \wtil{\fM}^w \to \fM^w$ be the composition of blow-up morphisms $\pi_1,\dots, \pi_N$.
Note that $\wtil{\pi}$ does not depend on choices of complete orders.

\medskip

At $i$-th step of blow-up, we have the node-identifying morphism $\iota_{\mu}^i$ obtained by the proper transform of the original node-identifying morphism $\iota_{\mu}$.  
We will get $\iota^N_\mu = \wtil{\iota}_\mu$.
Now, we want to see how $N_{\iota_\mu}$ \eqref{Normal:First} is related to the normal bundle of $\iota_{\mu}^i$ by an induction on $i$, which finally explains \eqref{Comp:NormalB}.

\subsubsection*{Initial step of the induction}

The following argument checks the conditions of \cite[Lemma 3.5]{VZ08} which is the key idea for the inductive argument.
We introduce some notations first.
We define an index set $A(\mu_a,\mu_b)$ for any $\mu_a, \mu_b \in \fS$,  
\begin{align*}
&A(\mu_a, \mu_b) := \left\{ \rho : [\ell(\mu_b)] \surra [\ell(\mu_a)] \left| 
\begin{array}{l}
K_i(\mu_b) \subset K_{\rho(i)}(\mu_a), \\
\sum_{j' \in \rho^{-1}(j)} d_{j'}(\mu_b) = d_j(\mu_a)  
\end{array}
\right.
\right\}
\end{align*} 
which is a more finer index set than the one in \cite[Section 4.2]{VZ08}.
Let 
\begin{itemize}
\item $I_j(\rho) \  := \rho^{-1}(j)$, 
\item $K_j(\rho) :=  K_j(\mu_b) \setminus (\cup_{j' \in I_j(\rho)} K_{j'}(\mu_a))$,
\item $I_0(\rho) \ := \rho^{-1}( \{ j \in [\ell(\mu_a)] : \ |K_j(\mu_a) \sqcup I_j(\rho)|=1, \  K_j(\rho)=\varnothing \} )$, and 
\item $K_0(\rho) := K_0(\mu_b) \setminus ( \cup_{j \in [\ell(\mu_a)]} K_j(\mu) )$.
\end{itemize}
Note that $|K_j(\mu_a) \sqcup I_j(\rho)| \geq 2$ for $j \in [\ell(\mu_a)] \setminus I_0(\rho)$.
Geometrically, each $\rho$ corresponds to an intersection component of $\fM^{\bar{\mu}_a}$ and $\fM^{\bar{\mu}_b}$.
Recall that each generic element in $\fM^{\bar{\mu}_a}$ has a genus one component $C_0$ in $\bar{\cM}_{1, K_0(\mu_a) }$ glued to genus zero curves $C_1,\dots, C_{\ell(\mu_a) }$ with marked points corresponding to elements in $K_j(\mu_a)$ through the nodal point $j \in [\ell(\mu_a)]$.
This element considered as in the intersection component corresponding to $\rho$ has the following property. 
For each $j \in [\ell(\mu_a)]$, $C_j$ has another $|I_j(\rho)|$-nodal points if $|K_j(\mu_a) \sqcup I_j(\rho)| \geq 2$, and each attached component has $K_j(\mu_a)$ marked points.
$K_j(\rho)$ denotes the remaining marked points on $C_j$.
We leave the interpretations of $K_0(\rho)$ and $I_0(\rho)$ for the readers.

For simplicity, we denote $\fM_{0,\mu} := \prod\limits_{i \in [\ell]} \fM^w_{0, \bullet\sqcup K_i,d_i }$.
For any $i < N(\bar{\mu})$, we have
\begin{align}\label{VZcondition1}
 \iota_{\mu}^{-1}(\fM^{\bar{\mu}_i}) = \bigsqcup\limits_{\rho \in A(\bar{\mu}_i, \mu) } \bar{\cM}_{1,\rho} \times \fM_{0,\mu}
 \end{align}
where 
$$A(\bar{\mu}_a, \mu_b) := \sqcup _{\{ \mu' | \bar{\mu'}= \bar{\mu}_a \}} A(\mu' ,\mu_b),$$ and
$\bar{\cM}_{1,\rho} \subset \bar{\cM}_{1, K_0(\mu)\sqcup [\ell(\mu)]}$ is an image of the node-identifying morphism
\[
\bar{\cM}_{1, K_0(\rho)\sqcup I_0(\rho)} \times \prod\limits_{j \in [\ell(\mu_i)]\setminus \rho(I_0(\rho))} \bar{\cM}_{0, \bullet\sqcup I_j(\rho)\sqcup K_j(\rho)} \to \bar{\cM}_{1, K_0(\mu)\sqcup [\ell(\mu)]}.
\]
And we have 
\begin{align}\label{VZcondition2}
\iota_{\mu}^*\left( N_{\iota_{\mu}( \bar{\cM}_{1,\rho} \times \fM_{0,\mu} ) / \fM^{\bar{\mu}_i} } \right) = \bigoplus\limits_{i \in [\ell]\setminus I_0(\rho)} L_i \boxtimes L_{\bullet,i} |_{\bar{\cM}_{1,\rho}\times \fM_{0,\mu} }.
\end{align}

Let
\begin{itemize}
\item
$L_{i,0} := L_i$
\item
$\bar{\cM}^0_{1,(K_0,[\ell])}:= \bar{\cM}_{1,K_0\sqcup [\ell]}$
\item
$\bar{\cM}^1_{1,(K_0,[\ell])}$ : The blow-up of $\bar{\cM}^0_{1,(K_0,[\ell])}$ along the locus $ \bigsqcup\limits_{\rho \in A(\mu_1 ,\mu)} \bar{\cM}^0_{1,\rho} $ where $\bar{\cM}^0_{1,\rho} := \bar{\cM}_{1,\rho}$. Note that we can check $\bar{\cM}^0_{1,\rho}$ are disjoint to each other by \cite[Lemma 2.6]{VZ08}.
\item
$D_{\rho}$ : Exceptional divisors of the blow-up $\pi_1$ supported on the locus $\bar{\cM}^0_{1,\rho}$.
\item
By abuse of notation, we again denote $\pi_1$ by the blow-up morphism $\bar{\cM}^1_{1,(K_0,[\ell])} \to \bar{\cM}^0_{1,(K_0,[\ell])}$.
\end{itemize}
We note that the collection of unramified morphisms $\{ \iota_{\mu_1},\dots,\iota_{\mu_N} \}$ is properly self-intersecting (see \cite[Definition 3.2]{VZ08} for the definition) and the image of $\iota_{\mu_1}$ is smooth by \cite[Section 4.3, (I8)]{VZ08}. Hence we can apply \cite[Corollary 3.4]{VZ08} to guarantee that the induced node-identifying morphism
\[
\iota^1_{\mu} : \bar{\cM}^1_{1,(K_0,[\ell])} \times \fM_{0,\mu} \to \fM^{|1}
\]
is again an unramified morphism whose image is an open substack of $\fM^{\bar{\mu}|1}$.
By \eqref{VZcondition1} and \eqref{VZcondition2}, we can apply \cite[Lemma 3.5]{VZ08} so that we have
\begin{align*}
N_{\iota_{\mu}^1} = \bigoplus\limits_{i \in [\ell]} \pi_1^*L_{i,0} \left( -\sum\limits_{\rho \in A(\bar{\mu}_1,\mu) | i \in I_0(\rho)} D_{\rho} \right) \boxtimes L_{\bullet,i}.
\end{align*}

\subsubsection*{Second step of the induction: $j< N(\bar{\mu})$ case}

Inductively, when $j < N(\bar{\mu})$, we have the following by applying \cite[Corollary 3.4]{VZ08} and \cite[Lemma 3.5]{VZ08} repeatedly.
\begin{itemize}
\item For $i > j$ the induced node-identifying morphisms
$$
\iota_{\mu_i}^j : \bar{\cM}^j_{1,(K_0(\mu_i),[\ell(\mu_i)])} \times \fM_{0,\mu_i} \to \fM^{|j}
$$ 
are unramified, and the collection $\{\iota_{\mu_{j+1}}^j,\dots \iota_{\mu_N}^j\}$ are properly self-intersecting. Note that the image of $\iota_{\mu_{j+1}}^j$ is smooth by \cite[Section 4.3, (I8)]{VZ08}.
\item
The normal bundle is decomposed into
\begin{align*}
N_{\iota_{\mu}^j} = \bigoplus\limits_{i \in [\ell]} \pi_j^*L_{i,j-1} \left( -\sum\limits_{\rho \in A(\bar{\mu}_j, \mu) | i \in I_0(\rho)} D_{\rho} \right) \boxtimes L_{\bullet,i}.
\end{align*}
\end{itemize}
Here are the notations above which are defined inductively
\begin{itemize}
\item
$\bar{\cM}^{j-1}_{1,\rho}$ : a proper transform of $\bar{\cM}^{j-2}_{1,\rho}$ via the blow-up $$\pi_{j-1} : \bar{\cM}^{j-1}_{1,(K_0,[\ell])} \to \bar{\cM}^{j-2}_{1,(K_0,[\ell])}.$$
\item
$\bar{\cM}^j_{1,(K_0,[\ell])}$ : the blow-up of $\bar{\cM}^{j-1}_{1,(K_0,[\ell])}$ along the locus $\bigsqcup\limits_{\rho \in A(\mu_j,\mu)} \bar{\cM}^{j-1}_{1,\rho}$. 
\item
$\pi_j : \bar{\cM}^{j}_{1,(K_0,[\ell])} \to \bar{\cM}^{j-1}_{1,(K_0,[\ell])} $ : the blow-up morphism.
\item
$D_{\rho}$ : Exceptional divisors of the blow-up $\pi_{j}$ supported on the locus $\bar{\cM}^{j-1}_{1,\rho}$.
\end{itemize}

We can check the conditions of \cite[Lemma 3.5]{VZ08} that we need to apply repeatedly. For $j < N(\bar{\mu})$, we have
\[
(\iota^j_{\mu})^{-1}(\fM^{\bar{\mu}_i|j}) = \bigsqcup\limits_{\rho \in A(\bar{\mu}_i,\mu)} \bar{\cM}^j_{1,\rho} \times \fM_{0, \mu}
\]
where $\bar{\cM}^j_{1,\rho}$ is the proper transform of $\bar{\cM}^{j-1}_{1,\rho}$ via the blow-up morphism $\pi_j$. Also, when $j < N(\bar{\mu})$ we have the following for each $\rho \in A(\mu_{j+1}, \mu)$
\begin{equation*}
(\iota^j_{\mu})^*\left( N_{\iota_{\mu}^j( \bar{\cM}^j_{1,\rho} \times \fM_{0,\mu} ) / \fM^{\bar{\mu}_{j+1}|j} } \right) = \bigoplus\limits_{i \in [\ell]\setminus I_0(\rho)} L_{i,j} \boxtimes L_{\bullet,i} |_{\bar{\cM}^j_{1,\rho}\times \fM_{0,\mu} }
\end{equation*} 
where $L_{i,j} := \pi_j^* L_{i,j-1} \left( -\sum\limits_{\rho \in A(\bar{\mu}_j,\mu) | i \in I_0(\rho)} D_{\rho} \right)$.

\subsubsection*{Third step of the induction: $j = N(\bar{\mu})$ case}

For simplicity, we denote 
$$ \wtil{\cM}_{1,(K_0,[\ell])} := \bar{\cM}^{ N(\bar{\mu})-1}_{1,(K_0,[\ell])},  \ \  \wtil{L}_i := L_{i,N(\bar{\mu})-1}.$$
In \cite{VZ08}, Vakil-Zinger explained that all $\wtil{L}_i$ are isomorphic, so we denote it by $\LL$.
Let $\EE$ be the Hodge bundle on $\bar{\cM}_{1, K_0\sqcup [\ell]}$. 
Let $\EE_0 : = \EE$ and we define $\EE_i$ inductively by $\EE_i := \pi_i^* \EE_{i-1} \left( \sum\limits_{\rho \in A(\bar{\mu}_i,\mu)} D_{\rho} \right)$. Let $\wtil{\EE} := \EE_{N(\bar{\mu})-1}$. Thus we obtain
\begin{align}\label{hodgetwisting}
\wtil{\EE} = \pi^* \EE \left( \sum\limits_{\rho \in \bigsqcup_{i=1}^{N(\bar{\mu})-1} A(\bar{\mu}_i,\mu)}\wtil{D}_{\rho} \right)
\end{align}
where $\pi : \wtil{\cM}_{1,(K_0,[\ell])} \to \bar{\cM}_{1,(K_0,[\ell])}$ is the composition of the blow-ups $\pi_1,\dots,\pi_{N(\bar{\mu})-1}$, $\wtil{D_{\rho}}$ are pull-backs of $D_{\rho}$.
Note that $\LL \cong \wtil{\EE}\dual$; see \cite{VZ08}.

Now we define $\PP\fM_{0,\mu}^0 := \PP(\oplus_{i  \in [\ell]} L_{\bullet,i})$, where $\oplus_{i  \in [\ell]} L_{\bullet,i}$ is a vector bundle on $\fM_{0,\mu}$. Let $\gamma$ be the tautological bundle on $\PP\fM_{0,\mu}^0$. 
Then we have the following natural induced morphism
\begin{multline*}
\iota^{N(\bar{\mu})}_{\mu} : \wtil{\cM}_{1,(K_0,[\ell])} \times \PP\fM^0_{0,\mu} \cong \PP( N_{\iota_{\mu}^{N(\bar{\mu})-1}}) \\
\cong \PP( (\iota_{\mu}^{N(\bar{\mu})-1})^* N_{\iota^{N(\bar{\mu})-1}_{\mu}(\wtil{\cM}_{1,(K_0,[\ell])} \times \fM_{0,\mu})/\fM^{|N(\bar{\mu})-1} }) \\
\lra \PP(N_{\fM^{\bar{\mu}|N(\bar{\mu})-1}/ \fM^{|N(\bar{\mu})-1}}) \cong \fM^{\bar{\mu}|N(\bar{\mu})}.
\end{multline*}
Since $\fM^{\bar{\mu}|N(\bar{\mu})}$ is the exceptional divisor of $\Bl_{\fM^{\bar{\mu}|N(\bar{\mu})-1}} \fM^{|N(\bar{\mu})-1} \cong \fM^{|N(\bar{\mu})}$, the normal bundle $N_{\fM^{\bar{\mu}|N(\bar{\mu})} / \fM^{|N(\bar{\mu})}}$ is the tautological bundle of the projectivization $\PP(N_{\fM^{\bar{\mu}|N(\bar{\mu})-1} / \fM^{|N(\bar{\mu})-1}})$.
Thus we have
\begin{align*}
N_{\iota_{\mu}^{N(\bar{\mu})}} \cong \wtil{\EE} \boxtimes \gamma.
\end{align*}

\subsubsection*{Last step of the induction: $j > N(\bar{\mu})$ case}

An idea of the last step is again an induction on $j$.
The new inductive argument is pretty much similar to the first and second steps. 

First, we consider $j=N(\bar{\mu})+1$. 
By \cite[Section 4.3 (I8)]{VZ08}, the image of the unramified morphism $\iota_{\mu}^{N(\bar{\mu})}$, which is equal to $\fM^{\bar{\mu}|N(\bar{\mu})} \subset \fM^{|N(\bar{\mu})}$, is a  smooth divisor and the image of $\iota_{\mu_{N(\bar{\mu})+1}}^{N(\bar{\mu})}$ is smooth. 
Thus $\fM^{\bar{\mu}|N(\bar{\mu})}$ and $\fM^{\bar{\mu}'|N(\bar{\mu})}$ intersect transversally for $\bar{\mu} \neq \bar{\mu}'$. 
Hence we observe that the collection $\{ \iota_{\mu}^{N(\bar{\mu})}, \iota_{\mu_{N(\bar{\mu})+1}}^{N(\bar{\mu})}, \dots, \iota_{\mu_N}^{N(\bar{\mu})} \}$ of unramified morphisms is properly self-intersecting.
Before going further, we introduce some notations.
\begin{itemize}
\item For $\rho \in A(\mu,\mu_j)$, $j>N(\bar{\mu})$, and for $i \in [\ell(\mu)]$ let $\fM_{0,\rho_i} \subset \fM^w_{0,K_i,d_i}$ be the image of the node-identifying morphism
\[
\fM_{0,\bullet\sqcup K_i(\rho)\sqcup I_i(\rho)} \times \prod\limits_{i' \in I_i(\rho)} \fM^w_{0,\bullet \sqcup K_{i'}(\bar{\mu}_j), d_{i'}} \to \fM^w_{0,\bullet \sqcup K_i, d_i}. 
\]
\item Let $\fM_{0,\rho} := \prod\limits_{i \in [\ell]} \fM_{0,\rho_i} \subset \fM_{0,\mu} $.
\item Let $\PP\fM_{0,\rho}^0 := \PP(\oplus_{i \in [\ell]\setminus I_0(\rho)} L_{\bullet,i})$.
\end{itemize}
Note that $\PP\fM^0_{0,\rho}$ are disjoint to each other by \cite[Lemma 3.9]{VZ08}.
Now we can apply \cite[Corollary 3.4]{VZ08} so that we have the following induced unramified node-identifying morphism
\begin{align*}
\iota_{\mu}^{N(\bar{\mu})+1} : \wtil{\cM}_{1,(K_0,[\ell])} \times \PP\fM_{0,\mu}^1 \to \fM^{|N(\bar{\mu})+1}
\end{align*}
where $\PP\fM^1_{0,\mu}$ is the blow-up of $\PP\fM_{0,\mu}^0$ along the smooth locus $\bigsqcup\limits_{\rho \in A(\mu, \bar{\mu}_{N(\bar{\mu})+1})} \PP\fM^0_{0,\rho}$. 
The domain of $\iota_{\mu}^{N(\bar{\mu}+1)}$ has the product form because
\begin{align*}
(\iota^{N(\bar{\mu})}_{\mu})^{-1}(\fM^{\bar{\mu}_{N(\bar{\mu})+1}|N(\bar{\mu})}) = \bigsqcup_{\rho \in A(\mu, \bar{\mu}_{N(\bar{\mu})+1})} \wtil{\cM}_{1,(K_0,[\ell])} \times \PP\fM_{0,\rho}^0,
\end{align*}
where
$$A(\mu_a, \bar{\mu}_b) := \sqcup _{\{ \mu' | \bar{\mu'}= \bar{\mu}_b \}} A(\mu_a ,\mu');$$
see \cite[Section 4.3]{VZ08} for details.
Since $\fM^{\bar{\mu}|N(\bar{\mu})}$ is a smooth divisor in $\fM^{| N(\bar{\mu})}$, we can apply \cite[Lemma 3.5]{VZ08} to obtain
\[
N_{\iota_{\mu}^{N(\bar{\mu})+1}} = (id \times \pi_1')^* N_{\iota_{\mu}^{N(\bar{\mu})}},
\]
where $\pi'_1 : \PP\fM_{0,\mu}^1 \to \PP\fM_{0,\mu}^0$ is the blow-up morphism.

Next we consider $j > N(\bar{\mu})+1$. We check the following inductively.
\begin{itemize} 
\item
For $j > N(\bar{\mu})+1$, by by \cite[Section 4.3 (I5)]{VZ08}, we have
\[
(\iota_{\mu}^{j-1})^{-1}( \fM^{\bar{\mu}_j | j-1} ) = \bigsqcup\limits_{\rho \in A(\mu,\bar{\mu}_j) } \wtil{\cM}_{1,(K_0,[\ell])} \times \PP\fM_{0,\rho}^{j-N(\bar{\mu}) - 1}
\]
where $\PP\fM_{0,\rho}^{j-N(\bar{\mu})-1}$ are the proper transforms of $\PP\fM_{0,\rho}^{j-N(\bar{\mu})-2}$ via the blow-up morphism $\pi_{j-N(\bar{\mu})-1}'$.
\item
By \cite[Corollary 3.4]{VZ08}, the induced node-identifying morphism
\[
\iota_{\mu}^j : \wtil{\cM}_{1,(K_0,[\ell])} \times \PP\fM_{0,\mu}^{j-N(\bar{\mu})} \to \fM^{|j-N(\bar{\mu})}
\]
is unramified, where $\PP\fM^{j-N(\bar{\mu})}_{0,\mu}$ is the blow-up of $\PP\fM_{0,\mu}^{j - N(\bar{\mu}) - 1}$ along the smooth locus $\bigsqcup\limits_{\rho \in A(\mu,\bar{\mu}_j) } \PP\fM_{0,\rho}^{j-N(\bar{\mu}) - 1}$.
We denote 
$$\pi'_{j-N(\bar{\mu})} : \PP\fM^{j-N(\bar{\mu})}_{0,\mu} \to \PP\fM^{j-N(\bar{\mu})-1}_{0,\mu}$$ 
the blow-up morphism.

\item
The image of $\iota^j_{\mu_{j+1}}$ is smooth by \cite[Section 4.3 (I8)]{VZ08}.  
Also the image of $\iota_{\mu}^j$ is smooth because it is a blow-up of the image of $\iota_{\mu}^{j-1}$ by \cite[Lemma 3.3]{VZ08}.
Hence the collection $\{ \iota^j_{\mu}, \iota^j_{\mu_{j+1}}, \dots, \iota^j_{\mu_N} \}$ is properly self-intersecting.

\item The normal bundle of $\iota_{\mu}^j$ can be obtained by
\[
N_{\iota_{\mu}^{j}} = (\id \times \pi'_{j-N(\bar{\mu})})^* N_{\iota_{\mu}^{j-1}}.
\]
by \cite[Lemma 3.5]{VZ08}.
\end{itemize}

Let us define $\PP\wtil{\fM}_{0,\mu} := \PP\fM_{0,\mu}^N$ and $\PP\wtil{\fM}_{0,\rho} := \PP\fM_{0,\rho}^N$. Let $\wtil{\pi}' : \PP\wtil{\fM}_{0,\mu} \to \PP\fM_{0,\mu}^0$ be the composition of blow-up morphisms $\pi'_1,\dots,\pi'_{N-N(\bar{\mu})}$. Also we denote $\iota_{\mu}^N$ by $\wtil{\iota}_{\mu}$.
Then we have
\begin{equation}\label{normalbundlefinal}
N_{\wtil{\iota}_{\mu}} \cong \wtil{\iota}_{\mu}^{\; *} N_{\wtil{\fM}^{\bar{\mu}} / \wtil{\fM}^w} \cong  \wtil{\EE}\dual \boxtimes ( \wtil{\pi}')^* \gamma.
\end{equation}

\medskip

By letting 
$$A_1(\mu) := \bigsqcup\limits_{i=1}^{N(\bar{\mu})-1} A(\bar{\mu}_i,\mu), \ \ A_0(\mu) := \bigsqcup\limits_{i=N(\bar{\mu})+1}^{N}A(\mu, \bar{\mu}_i),$$
we can write
\begin{align}\label{normalbundletwisting2}
& \wtil{\iota}_{\mu}^{\, *} \left( N_{\wtil{\fM}^{\bar{\mu}}/ \wtil{\fM}^w }\left( \sum\limits_{  \bar{\mu}' \neq \bar{\mu} } \wtil{\fM}^{\bar{\mu}'} \cap \wtil{\fM}^{\bar{\mu}} \right) \right) \\ \nonumber
& = \wtil{\iota}_{\mu}^{\, *} N_{\wtil{\fM}^{\bar{\mu}}/ \wtil{\fM}^w }  \left( \sum\limits_{\rho \in A_1(\mu)} \wtil{\cM}_{1,\rho} \times \PP\wtil{\fM}_{0,\mu} + \sum\limits_{\rho' \in A_0(\mu)} \wtil{\cM}_{1,(K_0,[\ell])} \times \PP\wtil{\fM}_{0,\rho'}  \right) \\ \nonumber
& =  N_{\wtil{\iota}_\mu }  \left( \sum\limits_{\rho \in A_1(\mu)} \wtil{\cM}_{1,\rho} \times \PP\wtil{\fM}_{0,\mu} + \sum\limits_{\rho' \in A_0(\mu)} \wtil{\cM}_{1,(K_0,[\ell])} \times \PP\wtil{\fM}_{0,\rho'}  \right)  , 
\end{align}
where $\wtil{\cM}_{1,\rho}$ is the proper transform of $D_\rho$ defined in the first and second step of the induction for $\rho \in A_1(\mu)$.
Recall that $\wtil{D}_{\rho}$ is the pull-back of $D_\rho$.
We can prove that $\wtil{D}_{\rho} = \wtil{\cM}_{1,\rho}$.
For $\rho' \in A_0(\mu)$, let $\wtil{D}_{\rho'} := \PP\wtil{\fM}_{0,\rho'}$. 
Therefore, by using \eqref{Remark2}, \eqref{normalbundletwisting2}, \eqref{normalbundlefinal}, and \eqref{hodgetwisting} sequentially we obtain
\[
\wtil{\iota}_{\mu}^{\, *} N_{\wtil{\fM}^{rat}/\wtil{\fM}^w}|_{\wtil{\fM}^{\bar{\mu}}} \cong \EE\dual \boxtimes \wtil{\EE'}
\]
where $\wtil{\EE'} := (\wtil{\pi}')^* \gamma \left( \sum\limits_{\rho' \in A_0(\mu)} \wtil{D}_{\rho'} \right)$.

\subsection{Connection to perfect obstruction theory}\label{Relperf}
Let $\wtil{\cM}^{rat} := \cup_{\bar{\mu}} \wtil{\cM}^{\bar{\mu}}$ and 
$$\LL_{\bar{\mu}} := N_{\wtil{\fM}^{\bar{\mu}}/\wtil{\fM}}|_{\wtil{\cM}^{\bar{\mu}}}, \ \ \LL_{rat} := N_{\wtil{\fM}^{rat}/\wtil{\fM}}|_{\wtil{\cM}^{rat}}.$$ 
By \eqref{Remark2}, we have 
$$
\LL_{rat}|_{\wtil{\cM}^{\bar{\mu}}} = \LL_{\bar{\mu}}\left( \sum\limits_{\bar{\mu}' \in \fS, \bar{\mu}'\neq \bar{\mu}} \wtil{\cM}^{\bar{\mu}'} \cap \wtil{\cM}^{\bar{\mu}} \right).
$$
In Corollary \ref{obstructionmorphism2}, we obtain a morphism
\begin{equation}\label{eq:inducedmorphism1}
\LL_{rat} \to H^1(E_{\wtil{\cM}/\wtil{\fM}^w} \dual|_{\wtil{\cM}^{rat}})
\end{equation}
whose zero locus is $\wtil{\cM}^{\bar{\mu}} \cap \wtil{\cM}^{red}$. 
Note that we have 
\begin{align*}
\wtil{\iota}_{\mu, \PP^n}^{\; *}H^1(E_{\wtil{\cM}/\wtil{\fM}^w} \dual|_{\wtil{\cM}^{\bar{\mu}}}) \cong \EE\dual \boxtimes ev_{\bullet}^* T_{\PP^n}.
\end{align*}
The morphism  
\begin{equation}\label{eq:inducedmorphism2}
\wtil{\iota}_{\mu, \PP^n}^{\; *} \left( \LL_{rat}|_{\wtil{\cM}^{\bar{\mu}}} \right) \cong \EE\dual \boxtimes q_{\mu}^* \wtil{\EE'} \to \wtil{\iota}_{\mu, \PP^n}^{\; *} H^1(E_{\wtil{\cM}/\wtil{\fM}^w} \dual|_{\wtil{\cM}^{\bar{\mu}}})
\end{equation}
induced by \eqref{eq:inducedmorphism1} coincides with the morphism of vector bundles $id_{\EE\dual} \boxtimes \wtil{\cD}_0$ in \cite[p. 1236]{Zin08standard} after a certain perturbation of the moduli space $\wtil{\cM}_{1,(K_0,[\ell])} \times \wtil{M}_{0,\mu}(\PP^n,d)$. The perturbed spaces are the virtual fundamental classes in symplectic geometry.
Here is the explanation of notations above. 
\begin{itemize}
\item
$\wtil{M}_{0,\mu}(\PP^n,d) := \bar{M}_{0,\mu}(\PP^n,d) \times_{\fM_{0,\mu}} \PP \wtil{\fM}_{0,\mu}$ where $\bar{M}_{0,\mu}(\PP^n,d)$ is defined similar to $\bar{M}_{0,\mu}(Q,d)$.
\item
$\wtil{\iota}_{\mu, \PP^n}$ : The node-identifying morphism $\wtil{\cM}_{1,(K_0,[\ell])} \times \wtil{M}_{0,\mu}(\PP^n,d) \to \wtil{\cM}^{\bar{\mu}}$.
\item
$q_{\mu}$ : The forgetful morphism $\wtil{M}_{0,\mu}(\PP^n,d) \to \PP \wtil{\fM}_{0,\mu}$.
\end{itemize}
In \cite{Zin08standard}, Zinger proved the morphism $id_{\EE\dual} \boxtimes \wtil{\cD}_0$ is injective whereas the morphism \eqref{eq:inducedmorphism2} has zero locus $(\wtil{\iota}_{\mu, \PP^n})^{-1}(\wtil{\cM}^{red} \cap \wtil{\cM}^{\bar{\mu}})$.

\section{Proof of Theorem \ref{main}} \label{Sect:Pf}

\subsection{Chern class expression of $A_{k,d}^{\bar{\mu}}$}\label{Sect:Chern}

We begin this section with introducing some vector bundles on $\wtil{\cM}^{\bar{\mu}}$ (caution : this is not $\wtil{\cM}^{p, \bar{\mu}}$).
Let
\begin{align*}
&V_1 := \RR^1\pi_{*} ev^* \cO_{\PP^n}(1)^{\oplus(n+1)}, \ \
V_2 := \RR^1\pi_{*}\left(\bigoplus\limits_{i=1}^m ev^* \cO_{\PP^n}(-\deg f_i) \otimes \omega_{\pi}\right),\\
&V:=V_1\oplus V_2, \ \ N^{\bar{\mu}}:=\pi_*\left(\bigoplus\limits_{i=1}^m ev^* \cO_{\PP^n}(-\deg f_i) \otimes \omega_{\pi}\right).
\end{align*}
Note that $\rank V_1 = n+1$ and $\rank V_2 = d(\sum_i \deg f_i) +m$.
We further note that $\wtil{\cM}^{p, \bar{\mu}}$ is isomorphic to the total space of $N^{\bar{\mu}}$.

Let $V_1^p$, $V^p_2$, and $V^p$ be the pull-back vector bundles on $\wtil{\cM}^{p, \bar{\mu}}$. 
Let
$\sigma_1 :  V_{1}^{p}  \to \cO_{\wtil{\cM}^{p,\bar{\mu}}}$ and $\sigma_2 :  V_{2}^{p}  \to \cO_{\wtil{\cM}^{p,\bar{\mu}}}$ be cosections defined by
\begin{align} \label{Cosection:Def}
& \sigma_1 : (u_0',\dots,u_n')  \mapsto \sum\limits_{i=1}^{m} p_i \sum\limits_{j=0}^{n} \frac{\rou f_i}{\rou u_j}(u_0,\dots,u_n)u_j' , \\ \nonumber
& \sigma_2 : (p_1',\dots,p_m')  \mapsto \sum\limits_{i=1}^{m} p_i'f_i(u_0, \dots, u_n),
\end{align}
and $\sigma := \sigma_1 \oplus \sigma_2$. 
Then the precise statement of Definition \ref{Def:CycleDecomp} can be written with the cosection $\sigma$, for instance 
\begin{align} \label{Amu:Simple}
A^{\bar{\mu}}_{k,d} = (-1)^{d\sum_i \deg f_i } (b_Q)_* 0^!_{H^1( E_{\wtil{\cM}^p/\wtil{\fB} }^\vee|_{\wtil{\cM}^{p, \bar{\mu}}} )  , \sigma} [C_{\wtil{\cM}^p/\wtil{\fB}}|_{\wtil{\cM}^{p, \bar{\mu}}}].
\end{align}
In this section, we will compute $A_{k,d}^{\bar{\mu}}$ in terms of Chern classes of vector bundles; see \eqref{virtcyclecomp}.
To do so, we first introduce some notations
\begin{itemize}
\item
$\PP_{\bar{\mu}} := \PP(N^{\bar{\mu}} \oplus \cO_{\wtil{\cM}^{\bar{\mu}}})$ be a completion of $\wtil{\cM}^{p, \bar{\mu}}$.
\item
$\bar{\gamma} : \PP_{\bar{\mu}} \to \wtil{\cM}^{\bar{\mu}}$ be the projection morphism.
\item
Let $D_{\infty} := \PP(N^{\bar{\mu}} \oplus 0) \subset \PP_{\bar{\mu}}$ be the divisor at infinity.
\item 
$\bar{V}^p_1 := \bar{\gamma}^*V_1(-D_{\infty})$, $\bar{V}^p_2:= \bar{\gamma}^*V_2$, $\bar{V}^p := \bar{V}^p_1\oplus \bar{V}^p_2$. 
\item 
$\bar{\sigma}_1$ and $\bar{\sigma}_2$ be induced cosections by $\sigma_1$ and $\sigma_2$ respectively. 
\item 
$\bar{\sigma} := \bar{\sigma}_1 \oplus \bar{\sigma}_2$.
\end{itemize}

\medskip

For simplicity, we denote $C_{\wtil{\cM}^p/\wtil{\fB}}|_{\wtil{\cM}^{p,\bar{\mu}}}$ by $C^{p,\bar{\mu}}$. 
Let $C^{p,\bar{\mu}}_b := C^{p,\bar{\mu}} \cap |0 \oplus V^p_2|$. 
By \cite[Proposition 5.1]{LO18}, we have
$$C^{p,\bar{\mu}}_b = C^{p,\bar{\mu}} \cap |0 \oplus V^p_2| \subset \wtil{\cM}^{p,\bar{\mu}}\cup |\gamma^*F|$$ 
where $\gamma : \wtil{\cM}^p \to \wtil{\cM}$ is the projection morphism, and $F$ is a rank $m$ subbundle $F \subset V_2|_{\Delta_{\bar{\mu}}}$, $\Delta_{\bar{\mu}} = \wtil{\cM}^{red} \cap \wtil{\cM}^{\bar{\mu}}$.
Let $R^{p,\bar{\mu}} := C_{C^{p,\bar{\mu}}_b / C^{p,\bar{\mu}}}$ be the normal cone to $C^{p,\bar{\mu}}_b$ in $C^{p,\bar{\mu}}$. 
Hence, we have $[C^{p,\bar{\mu}}] = [R^{p,\bar{\mu}}]$ in $A_*(V^p(\sigma))$. 
By \cite[Proposition 5.3]{LO18}, we obtain
\[
\Gysin{V^p,\sigma}[C^{p,\bar{\mu}}] = \Gysin{V^p,\sigma}[R^{p,\bar{\mu}}] = \bar{\gamma}_* \Gysin{\bar{V}^p_2,\bar{\sigma}_2} \cdot \Gysin{\bar{V}^p_1}[\bar{R^{p,\bar{\mu}}}].
\]
Here, $\Gysin{\bar{V}^p_1}$ is considered as a bivariant class to make sense of a homomorphism $\Gysin{\bar{V}^p_1} : A_*(\bar{V}^p(\bar{\sigma}) ) \to A_*( \bar{V}^p_2(\bar{\sigma}_2) )$ and $\bar{R^{p,\bar{\mu}}} $ is the closure of $R^{p,\bar{\mu}}$ in $\bar{V}^p$.

Let $B^{\bar{\mu}} := \Gysin{\bar{V}^p_1}[\bar{R^{\bar{p,\mu}}}]$.
Since $\bar{R^{p,\bar{\mu}}}$ is contained in $\PP^{\bar{\mu}} \cup |\bar{\gamma}^*F |$, we may regard
$B^{\bar{\mu}}$ as an element in the image of $A_*(\PP^{\bar{\mu}} \cup |\bar{\gamma}^*F|) \to A_*(\bar{V}^p_2(\bar{\sigma}_2))$. 
Then we can decompose $B^{\bar{\mu}}$ into $B^{\bar{\mu}} = B^{\bar{\mu}}_1 + B^{\bar{\mu}}_2$ such that $B^{\bar{\mu}}_1 \in A_*(\PP^{\bar{\mu}})$ and $B^{\bar{\mu}}_2 \in A_*(|\bar{\gamma}^*F|)$.
We can show that $\bar{\gamma}_* \Gysin{\bar{V}^p_2,\bar{\sigma}_2} B^{\bar{\mu}}_2 = 0$ by a dimension reason; see \cite[Section 5.0.1]{LO18}. 
Hence we have
\begin{align}\label{Gysineq1}
\Gysin{V^p,\sigma}[C^{p,\bar{\mu}}] = \bar{\gamma}_* \Gysin{\bar{V}^p_2,\bar{\sigma}_2} ( B^{\bar{\mu}}_1 ).
\end{align}

\smallskip

Let $C^{p,\bar{\mu}}_0 := C_{\wtil{\cM}^p/\wtil{\fM}^w}|_{\wtil{\cM}^{p,\bar{\mu}}}$. (Caution: $C^{p,\bar{\mu}}$ was $C_{\wtil{\cM}^p/\wtil{\fB}}|_{\wtil{\cM}^{p,\bar{\mu}}}$.)
Consider the gluing of \eqref{SESofH1:2}
\begin{align*} 
0 \to R^1\pi_*\cO_{\cC_{\wtil{\cM}^{p,\bar{\mu}}} } \stackrel{\alpha}{\lra} H^1( E_{\wtil{\cM}^p/\wtil{\fB} }^\vee|_{\wtil{\cM}^{p, \bar{\mu}}} ) \to H^1( E_{\wtil{\cM}^p/\wtil{\fM}^w }^\vee|_{\wtil{\cM}^{p, \bar{\mu}}} ) \to 0.
\end{align*}
Then the morphism $\alpha$ factors through 
\begin{align*}
\xymatrix{
& & R^1\pi_*\cO_{\cC_{\wtil{\cM}^{p,\bar{\mu}}} } \ar[ld]_-\alpha \ar[d]^-\alpha && \\
0 \ar[r] & V_1^p \ar[r] & H^1( E_{\wtil{\cM}^p/\wtil{\fB} }^\vee|_{\wtil{\cM}^{p, \bar{\mu}}} ) \ar[r] & V_2^p \ar[r] & 0.
}
\end{align*}
The quotient of $\alpha: R^1\pi_*\cO_{\cC_{\wtil{\cM}^{p,\bar{\mu}}} } \to V_1^p$ is $\gamma^*H^1( E_{\wtil{\cM}/\wtil{\fM}^w }^\vee|_{\wtil{\cM}^{\bar{\mu}}} )$. Let $V_1' := H^1( E_{\wtil{\cM}/\wtil{\fM}^w }^\vee|_{\wtil{\cM}^{\bar{\mu}}} )$. Then we can write $H^1( E_{\wtil{\cM}^p/\wtil{\fM}^w }^\vee|_{\wtil{\cM}^{p, \bar{\mu}}} ) = \gamma^* V_1' \oplus V^p_2$.
Note that $C^{p,\bar{\mu}}_0 \subset \gamma^* V_1' \oplus V_2^p$.

Let $(C^{p,\bar{\mu}}_0)_b := C^{p,\bar{\mu}}_0 \cap |0\oplus V^p_2|$ and $R^{p, \bar{\mu}}_0 := C_{(C^{p,\bar{\mu}}_0)_b/ C^{p,\bar{\mu}}_0}$.
The glueing of \eqref{SESofH1:2} gives rise to the exact sequence of abelian cones
$$
0 \to R^1\pi_*\cO_{\cC_{\wtil{\cM}^{p,\bar{\mu}}} } \to C^{p,\bar{\mu}} \to C^{p,\bar{\mu}}_0 \to 0,
$$
and in the middle term, we have $R^1\pi_*\cO_{\cC_{\wtil{\cM}^{p,\bar{\mu}}} } \cap C^{p,\bar{\mu}}_b = 0$ since $\alpha$ factors through $V_1^p$. 
Thus we have the following exact sequence of cones
\begin{align*}
\xymatrix@R=0pt@C=10pt{
 0 \ar[r] & R^1\pi_*\cO_{\cC_{\wtil{\cM}^{p,\bar{\mu}}} } \ar[r] & R^{p, \bar{\mu}} \ar[r] & R^{p, \bar{\mu}}_0 \ar[r] & 0, \\
 0 \ar[r] & R^1\pi_*\cO_{\cC_{\PP^{\bar{\mu}}} } \ar[r] & \bar{ R^{p, \bar{\mu}}} \ar[r] & \bar{ R^{p, \bar{\mu}}_0} \ar[r] & 0,
}
\end{align*}
where $\bar{ R^{p, \bar{\mu}}_0}$ is the closure of $R^{p, \bar{\mu}}_0$ in $\bar{\gamma}^*V'_1(-D_\infty) \oplus \bar{V}_2^p$.
Thus, we have 
\begin{align} \label{Gysineq2}
& B^{\bar{\mu}}_1 = \Gysin{\bar{V}^p_1}[\bar{R^{p, \bar{\mu}}|_{\wtil{\cM}^{p,\bar{\mu}}} }]  = \Gysin{ \bar{\gamma}^*V_1'(-D_{\infty})}[\bar{R^{p, \bar{\mu}}_0|_{\wtil{\cM}^{p,\bar{\mu}}} }] .
\end{align}

Now we study $R^{p, \bar{\mu}}_0$ by using a local computation.
We use the notations in Section \ref{Decomp:Coord}.
The normal cone $C^{p,\bar{\mu}}_0$ was locally represented by 
\begin{align*}
C^{p,\bar{\mu}}_0 &  \stackrel{loc}{\cong}  C_{Y/X\times \CC^{n+m} }|_{Y^i} \\
& \cong  \Spec\left(  \frac{ R/(t_i)[y_1\dots,y_{n+m}][x_1,\dots,x_{n+m}] }{(y_ix_j-y_jx_i)_{1\leq i<j \leq n+m} } \right).
\end{align*}
Then, the local equation of $(C^{p,\bar{\mu}}_0)_b$ is $\{ x_1=\dots = x_n=0 \}$.
Let 
$$\wtil{R} := \frac{ R/(t_i)[y_1\dots,y_{n+m}][x_1,\dots,x_{n+m}] }{(x_1,\dots,x_n, y_ix_j-y_jx_i)_{1\leq i<j \leq n+m} }$$
be the (local) coordinate ring of $(C^{p,\bar{\mu}}_0)_b$.
The normal cone $C_{(C^{p,\bar{\mu}}_0)_b/C^{p,\bar{\mu}}_0}$ is then
\begin{align*}
& C_{(C^{p,\bar{\mu}}_0)_b/C^{p,\bar{\mu}}_0}  \stackrel{loc}{\cong} \frac{\wtil{R}[z_1,\dots,z_n]}{(y_iz_j-y_jz_i)_{1\leq i\leq j \leq n}}.
\end{align*}
Recall that $(C^{p,\bar{\mu}}_0)_b \subset \wtil{\cM}^{p,\bar{\mu}} \cup \gamma^* F$. 
Since the local equation of $\wtil{\cM}^{p,\bar{\mu}}$ is $\{ x_1 = \dots = x_{n+m} = 0 \}$, we have
\begin{align*}
 C_{(C^{p,\bar{\mu}}_0)_b/C^{p,\bar{\mu}}_0}|_{\wtil{\cM}^{p,\bar{\mu}}} 
& \stackrel{loc}{\cong} \frac{\wtil{R}/(x_{n+1},\dots,x_{n+m})[z_1,\dots,z_n]}{(y_iz_j-y_jz_i)_{1\leq i\leq j \leq n}} \\
& \cong \frac{R/(t_i)[y_1,\dots,y_{n+m}][z_1,\dots,z_n]}{(y_iz_j-y_jz_i)_{1\leq i \leq j \leq n} }.
\end{align*}
From the above local computation, we observe that $R^{p, \bar{\mu}}_0|_{\wtil{\cM}^{p,\bar{\mu}}} = \gamma^* C^{\bar{\mu}}_0$ where
$C^{\bar{\mu}}_0 \subset V_1'$ is the coarse moduli space of the cone stack $\fC_{\wtil{\cM}/\wtil{\fM}^w} |_{\wtil{\cM}^{\bar{\mu}}} \subset \bdst{\EE_{\wtil{\cM}/\wtil{\fM}^w} \dual|_{\wtil{\cM}^{\bar{\mu}}}}$
which is irreducible.
Thus, we have 
\begin{align} \label{Gysineq3}
\Gysin{ \bar{\gamma}^*V_1'(-D_{\infty})}[\bar{R^{p, \bar{\mu}}_0|_{\wtil{\cM}^{p,\bar{\mu}}} }] = \Gysin{\bar{V'}^p_1}[\bar{\gamma^* C^{\bar{\mu}}_0}].
\end{align}

\medskip

Recall that $\Delta_{\bar{\mu}} = \wtil{\cM}^{red} \cap \wtil{\cM}^{\bar{\mu}}$.
Let $\Delta^p_{\bar{\mu}} := \Delta_{\bar{\mu}} \times_{\wtil{\cM}^{\bar{\mu}}} \wtil{\cM}^{p,\bar{\mu}} $. Let 
$$q : \what{\cM}^{\bar{\mu}}:=\Bl_{\Delta_{\bar{\mu}}}\wtil{\cM}^{\bar{\mu}} \to \wtil{\cM}^{\bar{\mu}}$$ 
be the blow-up morphism.
By abuse of notation, we denote the induced morphism $\what{\cM}^{\bar{\mu}} \times_{\wtil{\cM}^{\bar{\mu}} } \wtil{\cM}^{p,\bar{\mu}} \to \wtil{\cM}^{p,\bar{\mu}}$ by $q$ and let $D' \subset \what{\cM}^{\bar{\mu}} \times_{\wtil{\cM}^{\bar{\mu}} } \wtil{\cM}^{p,\bar{\mu}}$ be the exceptional divisor. 
By Corollary \ref{obstructionmorphism2}, we obtain an injective morphism of vector bundles 
$$i :  q^*\wtil{\LL}_{\bar{\mu}} (D') \hra q^*H^1(E_{\wtil{\cM}/ \fM^w}\dual|_{\wtil{\cM}^{\bar{\mu}}}) = q^*V_1'$$ 
where $\wtil{\LL}_{\bar{\mu}} := \LL_{rat}|_{\wtil{\cM}^{\bar{\mu}}}$. 
Moreover, $C^{\bar{\mu}}_0$ is equal to $q(\image i)$ by \eqref{Remark1}. 
Consider the diagram
\[
\xymatrix{
\what{\PP}_{\bar{\mu}}:= \PP_{\bar{\mu}} \times_{\wtil{\cM}^{\bar{\mu}} } \what{\cM}^{\bar{\mu}} \ar[rr]^-{q} \ar[dd]^-{\bar{\gamma}} & & \PP_{\bar{\mu}} \ar[dd]^(0.3){\bar{\gamma}} & \\
& \what{\cM}^{p,\bar{\mu}}:=\wtil{\cM}^{p,\bar{\mu}} \times_{\wtil{\cM}^{\bar{\mu}} } \what{\cM}^{\bar{\mu}} \ar[rr]^(0.5)q \ar[dl]_{\gamma} \ar@{^(->}[lu]_-{\mathrm{open}} & & \wtil{\cM}^{p,\bar{\mu}} \ar@{^(->}[ul]_-{\mathrm{open}} \ar[dl]_{\gamma} \\
\what{\cM}^{\bar{\mu}} \ar[rr]^-q & & \wtil{\cM}^{\bar{\mu}} &
.}
\]
By abuse of notation, we used notations $q$, $\gamma$, and $\bar{\gamma}$ for pull-backs.
Then we observe
\[
\bar{\gamma^* C^{\bar{\mu}}_0} = q( \image j )
\]
where $j$ is the injective morphism induced by $i$
\[
j : \bar{ \gamma^* (  q^*  \wtil{\LL}_{\bar{\mu}} (D') ) } =\bar{\gamma}^* (  q^*  \wtil{\LL}_{\bar{\mu}} (D'))(-q^*D_{\infty}) \hra \bar{\gamma}^* q^* V_1'(-q^* D_{\infty}) = q^* \bar{V_1'}.
\]
Hence we have
\begin{align} \label{Gysineq4}
[\bar{\gamma^*C^{\bar{\mu}}_0}] = q_* [ \bar{\gamma}^* (  q^*  \wtil{\LL}_{\bar{\mu}} (D'))(-q^*D_{\infty}) ].
\end{align}
From \eqref{Gysineq2}, \eqref{Gysineq3} and \eqref{Gysineq4}, we have 
\begin{align*}
B_1^{\bar{\mu}} & = \Gysin{\bar{\gamma}^*(V_1')(-D_{\infty})}[\bar{\gamma^* C^{\bar{\mu}}_0 }] \\
& = \Gysin{\bar{\gamma}^*(V_1')(-D_{\infty})}q_* [ \bar{\gamma}^* (  q^*  \wtil{\LL}_{\bar{\mu}} (D'))(-q^*D_{\infty}) ] \\
& = q_* c_{top} \left( \bar{\gamma}^* q^* V_1' (-q^* D_{\infty}) / \bar{\gamma}^*( q^*\wtil{\LL}_{\bar{\mu}}(D') )( -q^* D_{\infty}) \right)
\end{align*}
where $c_{top}(-)$ stands for the top Chern class.
Thus we have
\begin{align} \label{Eq:B1}
\bar{\gamma}_* B^{\bar{\mu}}_1 & = \bar{\gamma}_* q_* c_{top} \left( \bar{\gamma}^* q^* V_1' (-q^* D_{\infty}) / \bar{\gamma}^*( q^*\wtil{\LL}_{\bar{\mu}}(D') )( -q^* D_{\infty}) \right) \\ \nonumber
& = q_* \left( c\left( q^*V_1' / \bar{\gamma}^*( q^*\wtil{\LL}_{\bar{\mu}}(D') ) \right) s( (q^* N^{\bar{\mu}} ) \dual) \right)_{\rank V_1' - m - 1 }
\end{align}
by \cite[Example 3.2.2]{Ful98}.
Here $c(-)$ and $s(-)$ denote the Chern and Segre classes respectively, and $(-)_i$ indicates the degree $i$ component.
For more specific computation, we need the following lemma with a general situation.

\begin{lemm}\label{excpushforward}
Let $Y$ be a smooth variety and let $X\subset Y$ be a smooth subvariety of codimension $r$. Let $\pi : \wtil{Y}=Bl_X Y \to Y$ be the blow-up morphism and $D$ be the exceptional divisor.
Let $E$ and $L$ be a vector bundle and a line bundle on $Y$ respectively. 
Assume that a line bundle $\pi^*L(D)$ is embedded in $ E$ as a subbundle.
Then we have 
\[
\pi_* c_{r-1}(\pi^*E/\pi^*L(D)) = \left( \frac{c(E)}{c(L)} \right)_{r-1}.
\]
\end{lemm}
\begin{proof}
We have
\begin{align*}
c_{r-1}(\pi^*E/\pi^*L(D)) & = \left( \frac{c(\pi^*E)}{c(\pi^*L(D))} \right)_{r-1} \\
& = \left( \frac{c(\pi^*E)}{1 + \pi_*c_1(L) + D} \right)_{r-1} \\
& = \left( c(\pi^*E)\left( \sum_{i=0}^\infty (-1)^i (\pi^*c_1(L)+D)^i\right) \right)_{r-1}.
\end{align*}
Note that $D \cong \PP(N_{X/Y})$. For every $1 \leq l \leq k \leq r-1$, we have
\begin{align*}
 \pi_* \left( c_{r-1-k}(\pi^*E)\cdot D^{l} \right) & = c_{r-1-k}(E) \cdot (\pi_*D^{l})  \\
& = c_{r-1-k}(E) \cdot \pi_*( c_1(\cO_{\PP(N_{X/Y})}(-1))^{l-1} \cdot D) \\
& = (-1)^{l-1} \cdot c_{r-1-k}(E) \cdot s_{l-r}(N_{X/Y}) [X] = 0.
\end{align*}
Hence we have
\begin{align*}
&\pi_* \left( c(\pi^*E)\left( \sum_{i=0}^\infty (-1)^i (\pi^*c_1(L)+D)^i\right) \right)_{r-1} \\ 
& = \pi_* \left( c(\pi^*E)\left( \sum_{i=0}^\infty (-1)^i (\pi^*c_1(L))^i\right) \right)_{r-1} = \left( \frac{c(E)}{c(L)} \right)_{r-1}.
\end{align*}
\end{proof}

By \eqref{Gysineq1}, \eqref{Eq:B1} and Lemma \ref{excpushforward}, we have
\begin{align}\label{virtcyclecomp}
\Gysin{V^p,\sigma}[C^{p,\bar{\mu}}] & = \bar{\gamma}_* \Gysin{\bar{V}^p_2,\bar{\sigma}_2} ( B^{\bar{\mu}}_1 ) \\ \nonumber
& = \Gysin{V_2,\xi_2}\bar{\gamma}_* ( B^{\bar{\mu}}_1 ) \\ \nonumber
& = \Gysin{V_2,\xi_2}[\wtil{\cM}^{\bar{\mu}}] \left( \frac{c(V_1')s((N^{\bar{\mu}})\dual) }{c(\wtil{\LL}_{\bar{\mu}}) } \right)_{\rank V_1' - m - 1}
\end{align}
where the cosection $\xi_2 : V_{2} \to \cO_{\wtil{\cM}^{\bar{\mu}}}$ is defined by
\begin{align*}
& (\dot{p_1},\dots,\dot{p_m}) \mapsto \sum\limits_{i=1}^{m} \dot{p_i}f_i(u_0, \dots, u_n), \ u_0,\dots,u_n \in \Gamma(\cC_{\wtil{\cM}^{\bar{\mu}}}, ev^* \cO_{\PP^n}(1)).
\end{align*}
By \eqref{Amu:Simple} and \eqref{virtcyclecomp}, we have
\begin{align}\label{chernclassexpress1}
A^{\bar{\mu}}_{k,d} = (-1)^{d\sum_i \deg f_i } (b_Q)_* \left( \Gysin{V_2,\xi_2}[\wtil{\cM}^{\bar{\mu}}] \left( \frac{c(V_1')s((N^{\bar{\mu}})\dual) }{c(\wtil{\LL}_{\bar{\mu}}) } \right)_{\rank V_1' - m - 1} \right).
\end{align}

\subsection{Condition (2) of Theorem \ref{main} : Computation of $A^{\bar{\mu}}_{k,d}$ in terms of Chern classes of vector bundles and Zinger's formula} \label{ChernClass}

Let
\begin{itemize}
\item
$\wtil{M}_{0,\mu}(Q,d) : = \bar{M}_{0,\mu}(Q,d) \times_{\fM_{0,\mu} } \PP\wtil{\fM}_{0,\mu}$.
\item
$\wtil{\cM}^{\bar{\mu}}(Q) := \wtil{\cM}^{\bar{\mu}} \times_{\bar{M}_{1,k}(\PP^n,d) } \bar{M}_{1,k}(Q,d)$.
\end{itemize}
The node-identifying morphism $\wtil{\iota}_{\mu}$ lifts to the following morphisms
\begin{align*}
& \wtil{\iota}_{\mu,\PP^n} : \wtil{\cM}_{1,(K_0,[\ell])} \times \wtil{M}_{0,\mu}(\PP^n,d) \to \wtil{\cM}^{\bar{\mu}}, \\
& \wtil{\iota}_{\mu,Q} : \wtil{\cM}_{1,(K_0,[\ell])} \times \wtil{M}_{0,\mu}(Q,d) \to \wtil{\cM}^{\bar{\mu}}(Q).
\end{align*}
We define a subgroup $S(\mu)$ of the symmetric group $S_{\ell}$ by
\begin{align*}
S(\mu) := & \{ f \in S_{\ell} \ | \ d_i = d_{f(i)} \textrm{ and } |K_i|=|K_{f(i)}| \textrm{ for all } i \in [\ell] \}.
\end{align*}
Note that $S(\mu)$ acts on $\wtil{\cM}_{1,(K_0,[\ell])} \times \wtil{M}_{0,\mu}(\PP^n,d)$ (resp. $\wtil{\cM}_{1,(K_0,[\ell])} \times \wtil{M}_{0,\mu}(Q,d)$) by permuting marked points on $\wtil{\cM}_{1,(K_0,[\ell])}$ and components in $\wtil{M}_{0,\mu}(\PP^n,d)$ (resp. $\wtil{M}_{0,\mu}(Q,d)$). Then $\wtil{\iota}_{\mu,Q}$ and $\wtil{\iota}_{\mu,\PP^n}$ send an $S(\mu)$-orbit to a point.
It implies that 
$$\deg(\wtil{\iota}_{\mu,\PP^n}) = \deg(\wtil{\iota}_{\mu,Q}) =|S(\mu)|.$$
Let us define 
\begin{itemize}
\item
$\wtil{M}_{0,\bar{\mu}}(\PP^n,d)  := \bigsqcup\limits_{\mu' \in \fS, \;  \bar{\mu'} = \bar{\mu} } \wtil{M}_{0,\mu'}(\PP^n,d)$. 
\item
$\wtil{M}_{0,\bar{\mu}}(Q,d)  := \bigsqcup\limits_{\mu' \in \fS, \;  \bar{\mu'} = \bar{\mu}} \wtil{M}_{0,\mu'}(Q,d).$ 
\item
$\PP\wtil{\fM}_{0,\bar{\mu}}  := \bigsqcup\limits_{\mu' \in \fS, \;  \bar{\mu'} = \bar{\mu}} \PP\wtil{\fM}_{0,\mu'}.$
\end{itemize}
Note that $|\{ \mu' \in \fS \ | \ \bar{\mu'} = \bar{\mu} \}| = \ell!/|S(\mu)|$.
Thus the induced node-identifying morphisms
\begin{align*}
\wtil{\iota}_{\bar{\mu},\PP^n} : \wtil{\cM}_{1,(K_0,[\ell])} \times \wtil{M}_{0,\bar{\mu}}(\PP^n,d) \to \wtil{\cM}^{\bar{\mu}} \\
\wtil{\iota}_{\bar{\mu},Q} : \wtil{\cM}_{1,(K_0,[\ell])} \times \wtil{M}_{0,\bar{\mu}}(Q,d) \to \wtil{\cM}^{\bar{\mu}}(Q)
\end{align*}
have degree $\ell !$. Hence we have
\begin{align}\label{Gysinpushforward1}
\ell! \cdot \Gysin{V_2,\xi_2}[\wtil{\cM}^{\bar{\mu}}] = (\wtil{\iota}_{\mu,Q})_* \Gysin{ (\wtil{\iota}_{\mu,\PP^n})^*V_2, (\wtil{\iota}_{\mu,\PP^n})^*\xi_2} \left( [\wtil{\cM}_{1,(K_0,[\ell])}] \times [\wtil{M}_{0,\bar{\mu}}(\PP^n,d)] \right).
\end{align}
by the bivariant property of localized Gysin map.
Therefore we have
\begin{align}\label{chernclassexpression3}
\nonumber
A^{\bar{\mu}}_{k,d} & = (-1)^{d\sum_i \deg f_i } (b_Q)_* \left( \Gysin{V_2,\xi_2}[\wtil{\cM}^{\bar{\mu}}]  \frac{c(V_1')s((N^{\bar{\mu}})\dual) }{c(\wtil{\LL}_{\bar{\mu}}) } \right)_{\rank V_1' - m - 1}  \\ \nonumber
& = \frac{(-1)^{d\sum_i \deg f_i }}{\ell!}(b_Q)_*(\wtil{\iota}_{\mu,Q})_* \Gysin{ (\wtil{\iota}_{\mu,\PP^n})^*V_2, (\wtil{\iota}_{\mu,\PP^n})^*\xi_2} \left( [\wtil{\cM}_{1,(K_0,[\ell])}] \times [ \wtil{M}_{0,\bar{\mu}}(\PP^n,d)] \right) \\ \nonumber
& \left( \frac{c((\wtil{\iota}_{\mu,Q})^*V_1')s((\wtil{\iota}_{\mu,Q})^*(N^{\bar{\mu}})\dual) }{c((\wtil{\iota}_{\mu,Q})^*\wtil{\LL}_{\bar{\mu}}) } \right)_{\rank V_1' - m - 1} \\ \nonumber
& = \frac{ 1 }{\ell!}(b_Q)_* \left( [\wtil{\cM}_{1,(K_0,[\ell])}] \times [ \wtil{M}_{0,\mu}(Q,d)]\virt \right) \\ \nonumber
& \left( \frac{c(\EE\dual \boxtimes ev_{\bullet}^* T_{\PP^n})s(\EE\dual \boxtimes ev_{\bullet}^*( \oplus_i \cO_{\PP^n}(\deg f_i) ) ) }{c(\EE\dual \boxtimes q_{\bar{\mu}}^*\wtil{\EE}')} \right)_{\rank V_1' - m - 1} \\ 
& = \frac{ 1 }{\ell!}(b_Q)_* \left( [\wtil{\cM}_{1,(K_0,[\ell])}] \times [\wtil{M}_{0,\bar{\mu}}(Q,d)]\virt \right) \left( \frac{c(\EE\dual  \boxtimes ev_{\bullet}^* T_{Q})}{c(\EE\dual \boxtimes q_{\bar{\mu}}^*\wtil{\EE}')} \right)_{\rank V_1' - m - 1}
\end{align}
where $q_{\bar{\mu}} : \wtil{M}_{0,\bar{\mu}}(Q,d) \to \PP\wtil{\fM}_{0,\bar{\mu}}$ is the projection morphism.
The first equality comes from \eqref{chernclassexpress1}, the second equality comes from \eqref{Gysinpushforward1}, the third equality comes from \cite[Proposition 5.5]{LO18}, \eqref{eq:inducedmorphism2}, and the base change theorem, and the fourth equality comes from the exact sequence of tangent bundles
\[
0 \to T_Q \to T_{\PP^n}|_Q \to \oplus_i \cO_{\PP^n}(\deg f_i) |_Q \to 0.
\]
\eqref{chernclassexpression3} exactly coincides with the formula \cite[(3-29)]{Zin08standard}. Note that an integration over $A^{\bar{\mu}}_{k,d}$ induces the formula \cite[Theorem 1A]{Zin08standard}; see \cite[Section 3.4]{Zin08standard}.
So the second condition (2) of Theorem \ref{main} is explained.

\subsection{Condition (1) of Theorem \ref{main}} \label{Review:reducedcycle}
\cite[Theorem 2.11]{HL10} tells us that 
$$N^{red} := \pi_* ev^* ( \oplus_{i=1}^m \cO_{\PP^n}(\deg f_i) ) $$ 
is a vector bundle on $\wtil{\cM}^{red}$.
Let $s$ be the section on $N^{red}$ induced by $f_i$, which are defining equations of $Q \subset \PP^n$. 
By Definition \ref{Def:CycleDecomp} and \cite[Proposition 4.1]{LO18}, we have
$$
A^{red}_{k,d} = \Gysin{\bdst{E_{\wtil{\cM}^p/\wtil{\fB} }\dual }, \sigma }[\fC_{\wtil{\cM}^{p,red}/\wtil{\fB}}] = \Gysin{(N^{red})\dual, s^\vee}[\wtil{\cM}^{p,red}] .
$$ 
The most right-hand side $\Gysin{(N^{red})\dual,s^\vee}[\wtil{\cM}^{p,\red}]$ is a refined Euler class. Thus the first condition (1) of Theorem \ref{main} is explained.

\newpage

\section*{Notations}
\noindent The below is a table of notations frequently used.
\begin{small}
\begin{table}[ht]
\label{eqtable}
\renewcommand\arraystretch{1.0}
\noindent\[
\begin{array}{l|l}
\hline
Q & \text{a complete intersection in } \PP^n \text{ defined by } \{f_1=\dots=f_m=0\}  \\
\fS_{k,d} \text{ or } \fS & \text{the index set of elements $\mu = ( (d_1,K_1),\dots,(d_{\ell},K_{\ell}) )$} \\ & \text{such that $\sum_i d_i = d$, $K_j$ are mutually disjoint subsets of $[k]$} \\
\bar{\fS} & \text{the index set of elements $\bar{\mu} = \{ (d_1,K_1),\dots,(d_{\ell},K_{\ell}) \}$} \\
\fM_{g,k} \text{ or } \fM & \text{the moduli space of prestable genus $g$ curves with $k$ marked}\\
& \text{points} \\
\fM_{g,k,d}^w \text{ or } \fM^w& \text{the moduli space of prestable genus $g$, weight $d$ curves with} \\
& \text{$k$ marked points}\\
\fM^{\bar{\mu}} & \text{the closed substack of $\fM^w$ parametrizing $\bar{\mu}$-type weighted curves} \\
\fM_{0,\mu} & \prod\limits_{i \in [\ell]} \fM^w_{0, \bullet\sqcup K_i,d_i } \\
\wtil{\fM}^w & \text{Hu-Li's desingularization of $\fM^w_{1,k,d}$}\\
\wtil{\fM}^{\bar{\mu}} & \text{the exceptional divisor in $\wtil{\fM}^w$ lying on a proper transform of $\fM^{\bar{\mu}}$ }\\
\fM_{g,k,d}^{div} \text{ or } \fM^{div} & \text{the moduli space of prestable genus $g$ curves with $k$ marked} \\
& \text{points and a degree $d$ divisor on $C$}\\
\wtil{\fM}^{div} & \wtil{\fM}^w \times_{\fM_{1,k,d}^w} \fM^{div}_{1,k,d} \\
\fB_{g,k,d} \text{ or } \fB & \text{the moduli space of prestable genus $g$ curves with $k$ marked} \\
& \text{points and a degree $d$ line bundle on $C$}\\
\wtil{\fB} & \wtil{\fM}^w \times_{\fM_{1,k,d}^w} \fB_{1,k,d} \\
\bar{M}_{\bar{\mu}}(\PP^n,d) & \bar{M}_{1,k}(\PP^n,d) \times_{\fM^w} \fM^{\bar{\mu}} \\
\bar{M}_{\bar{\mu}}(Q,d) & \bar{M}_{1,k}(Q,d) \times_{\fM^w} \fM^{\bar{\mu}} \\
\bar{M}_{0,\mu}(\PP^n,d) & \bar{M}_{0,\bullet \sqcup K_1(\mu)}(\PP^n,d_1(\mu)) \times_{\PP^n} \times \dots \times_{\PP^n} \bar{M}_{0,\bullet \sqcup K_{\ell(\mu)}(\mu)}(\PP^n,d_{\ell(\mu)}(\mu)) \\
\bar{M}_{0,\mu}(Q,d) & \bar{M}_{0,\bullet \sqcup K_1(\mu)}(Q,d_1(\mu)) \times_Q \times \dots \times_Q \bar{M}_{0,\bullet \sqcup K_{\ell(\mu)}(\mu)}(Q,d_{\ell(\mu)}(\mu)) \\
\wtil{\cM} & \wtil{\fM}^w \times_{\fM^w_{1,k,d}} \bar{M}_{1,k}(\PP^n,d)\\
\wtil{\cM}^{red} & \text{the reduced component of $\wtil{\cM}$}\\
\wtil{\cM}^{\bar{\mu}} & \text{an irreducible component of $\wtil{\cM}$ indexed by $\bar{\mu}$} \\
\wtil{\cM}_Q & \wtil{\fM}^w \times_{\fM^w_{1,k,d}} \bar{M}_{1,k}(Q,d)\\ 
\wtil{\cM}^p & \wtil{\cM} \times_{\bar{M}_{1,k}(\PP^n,d)} \bar{M}_{1,k}(\PP^n,d)^p  \\
\wtil{\cM}^{p,red}  &\text{the reduced component of $\wtil{\cM}^p$} \\
\wtil{\cM}^{p,\bar{\mu}}  &\wtil{\cM}^{\bar{\mu}} \times_{\bar{M}_{1,k}(\PP^n,d)} \bar{M}_{1,k}(\PP^n,d)^p \\
\what{\cM}^{\bar{\mu}} & \text{the blow-up of $\wtil{\cM}^{\bar{\mu}}$ along $\wtil{\cM}^{\bar{\mu}} \cap \wtil{\cM}^{red}$} \\
\what{\cM}^{p,\bar{\mu}} & \wtil{\cM}^{p,\bar{\mu}} \times_{\wtil{\cM}^{\bar{\mu}}} \what{\cM}^{\bar{\mu}} \\
\iota_{\mu} & \text{the node-identifying morphism $\bar{\cM}_{1,K_0\sqcup [\ell]} \times \fM_{0,\mu} \to \fM^{\bar{\mu}}$.} \\
\wtil{\iota}_{\mu} & \text{the proper transform of $\iota_{\mu}$ along the blow-up $\wtil{\fM}^{\bar{\mu}} \to \fM^{\bar{\mu}}$.} \\
\pi: \cC \to X & \text{a universal curve on a stack $X$} \\
ev & \text{an evaluation morphism from $\cC$} \\
\fC_{A/B} & \text{the relative intrinsic normal cone of $A$ relative to $B$.} \\
C_{A/B} & \text{the coarse moduli space of the intrinsic normal cone $\fC_{A/B}$} \\
\bdst{E_0 \to E_1} & \text{a bundle stack $[E_1/E_0]$}\\
\Gysin{\bdst{E},\sigma} \text{ or } \Gysin{E,\sigma} & \text{a localized Gysin map}\\
\hline
\end{array}
\]
\end{table}
\end{small}

\end{document}